\documentclass[a4paper,10pt,abstract=true,parskip=true]{scrartcl}

\usepackage[british]{babel}
\usepackage{amsmath,amssymb,wasysym,amsfonts,amsthm,amsthm}
\usepackage{empheq}
\usepackage{bm}

\usepackage{graphicx,epsfig}
\usepackage{subfig}
\usepackage{overpic}

\usepackage{txfonts}
\usepackage[T1]{fontenc}
\usepackage{textcomp}
\usepackage{xcolor}
\usepackage{titletoc}

\usepackage[a4paper,margin=1.5in]{geometry}
\usepackage[colorlinks=true, allcolors=blue!30!black]{hyperref}
\usepackage{url,enumerate,framed,ragged2e}
\usepackage[inline]{enumitem}
\usepackage{authblk}

\newtheorem{thm}{Theorem}

\newtheorem{prop}[thm]{Proposition}

\newtheorem{res}[thm]{Result}
\theoremstyle{remark}
\newtheorem{rmk}[thm]{Remark}
\numberwithin{equation}{section}
\numberwithin{thm}{section}

\newcommand{\eps}{\varepsilon}

\providecommand{\keywords}[1]{\flushleft\textbf{Keywords:~} #1}

\definecolor{matb}{rgb}{0, 0.4470, 0.7410}
\definecolor{darkb}{rgb}{0, 0, 0.8}
\definecolor{darkg}{rgb}{0, 0.65, 0}
\definecolor{green1}{rgb}{0, 0.6390, 0.6390}
\definecolor{green2}{rgb}{0.4540, 0.9210, 0.8540}

 \title{Constructing far-from-equilibrium patterns in a cross-diffusion vegetation-autotoxicity model}
 \author{Annalisa Iuorio\thanks{Department of Engineering, Parthenope University of Naples}, Cinzia Soresina\thanks{Department of Mathematics, University of Trento}, Frits Veerman\thanks{Mathematical Institute, University of Leiden}}
 \date{}

\begin{document}

\maketitle

\begin{abstract} 
\noindent 
Using geometric singular perturbation theory, we construct stationary, periodic, front-type far-from-equilibrium patterns in a cross-diffusion model for plant growth under the influence of toxicity. We show how existing techniques for the analysis of far-from-equilibrium patterns in one spatial dimension can be extended to include cross-diffusion terms and prove the existence of a one-parameter family of these patterns. For a general cross-diffusion model class, we show when front-type patterns may appear depending on the shape of the critical manifold, which is determined by the specific properties of the reaction terms.
\end{abstract}


\justify


\section{Introduction}\label{sec:intro}
    
Vegetation patterns have been an increasingly active research area in the last few decades due to their importance in describing communities' proximity to catastrophic shifts, including desertification \cite{von_Hardenberg_2001, Kefi_2014} or, alternatively, strategies to promote ecosystem resilience \cite{Banerjee_2026, Bastiaansen_2020_2, Eigentler_2018, Rietkerk_2021}. These patterns include fronts, spots, gaps, and labyrinths, as well as vegetation bands \cite{GIV.2025, Klausmeier_1999, Meron_2012, Meron_2007}. Most studies identify biomass-water feedback as the main mechanism behind the formation of vegetation patterns, which in turn explains their abundance in arid and semi-arid environments \cite{Borgogno_2009, Dijkstra_2011, d2006patterns, Filippini2026, Gandhi_2023, Gilad_2007, Martinez_2023, Meron_2004, Rietkerk_2002}. However, in recent years, an additional biological factor has been revealed to play a prominent role both in the emergence and the structure of vegetation patterns, namely toxicity \cite{Bonanomi_2014, Carteni_2012, Consolo_2023, Ruiz_2023}. Focusing on autotoxicity, several biological mechanisms can explain it; here, we consider self-DNA due to litter decomposition as it is the most well-established to this day \cite{Mazzoleni_2014}. This crucial element in plant-soil negative feedback is able to promote biodiversity \cite{Mazzoleni_2010, Mazzoleni_2007} and induce dynamic, asymmetric patterns \cite{IuorioVeerman.2021, Marasco_2020, Marasco_2014}. Moreover, recent works have shown how stable spatial structures can be obtained also when considering only biomass-toxicity interactions -- allowing for a more general description of pattern formation also when water is not a scarce resource -- as long as cross-diffusion terms are considered~\cite{GIS_2026}.
    
Cross-diffusion was originally introduced at a phenomenological level in models such as the classical Shigesada--Kawasaki--Teramoto (SKT) system~\cite{SKT}, proposed to describe spatial segregation in interacting populations competing for the same resource. Only later was it recognised that such terms can be rigorously justified as effective macroscopic descriptions arising from underlying fast-reaction mechanisms~\cite{DesvillettesRDM, Desvillettes2015, IMN}, obtained by eliminating rapidly equilibrating variables and dichotomising states.
This perspective has since been successfully applied in a variety of contexts, including predator–prey systems~\cite{ConfortoDesSoresina2018, DesvillettesFiorentinoMautone, DesvillettesSoresina, IIR}, cockroach clustering~\cite{cockroaches}, starvation-driven diffusion~\cite{Brocchieri2025, BrocchieriDesvillettes}, models of the Neolithic transition~\cite{Elias2018}, metal corrosion processes~\cite{Corrosione}, and biomass–toxicity interactions~\cite{GIS_2026}.
Although this mechanism does not universally lead to pattern formation, in certain cases (most notably in the SKT model, as well as the metal surface corrosion and plant-autotoxicity models), the induced cross-diffusion term enables the emergence of stable spatial patterns even in the absence of a classical activator–inhibitor reaction scheme~\cite{Breden2019}. In particular, the structure of the Jacobian matrix and its interplay with the form of the cross-diffusion terms and the corresponding fast-reaction system play a crucial role in determining the onset of such instabilities~\cite{BrocchieriSoresina}.

A common approach in the investigation of pattern-forming systems is based on the combination of linear stability analysis, weakly nonlinear analysis, continuation, and numerical simulations. In this paper, the analysis will be based on \emph{geometric singular perturbation theory} (GSPT), see e.g.~\cite{Hek2010, Kuehn.2015}. The analytical tools provided by GSPT have been instrumental in the constructive existence and stability of far-from-equilibrium multiple-scale patterns in classes of reaction-diffusion systems \cite{DoelmanVeerman.2015, dRDR.2016}, including those describing the dynamics of biomass and autotoxicity effects \cite{Carter_etal.2024, GIV.2025, IuorioVeerman.2021}. GSPT uses the presence of a small parameter (usually denoted by $0<\eps \ll 1$) to distinguish multiple spatial scales in solutions to pattern-forming systems. These multiple spatial scales allow the decomposition of the underlying system into so-called ``fast'' and ``slow'' subsystems, which are of lower dimension than the full system, and are therefore more amenable to analysis. These subsystems are studied using dynamical systems techniques, identifying orbit structures and conserved quantities. The results of the analysis of these subsystems can then be combined to constructively prove the existence of pattern-like solutions to the full system. We note that this analysis is fundamentally different from a Turing analysis, which focuses on the initial stages of the \emph{formation} process of \emph{close-to-equilibrium} patterns. In contrast, GSPT gives access to the \emph{constructive existence} (and stability) of \emph{far-from-equilibrium} patterns; indeed, the nonlinear nature of the underlying system is crucial in the analysis, and will determine which multi-scale patterns will exist. As such, geometric singular perturbation theory has proven to be particularly useful to analyse the existence of front-type or spike-type patterns, such as those shown in Fig.~\ref{fig:numsim_horns}, in reaction-diffusion systems without cross-diffusion \cite{Brown_etal.2023, Carter_etal.2024, CRS.2021, CarterScheel.2018, DvHK.2008, DoelmanVeerman.2015, GIV.2025, Hek2010, IuorioVeerman.2021, NDV.2026, Rademacher.2013, dRDR.2016}.

\begin{figure}
    \centering
    \begin{overpic}[width=13cm]{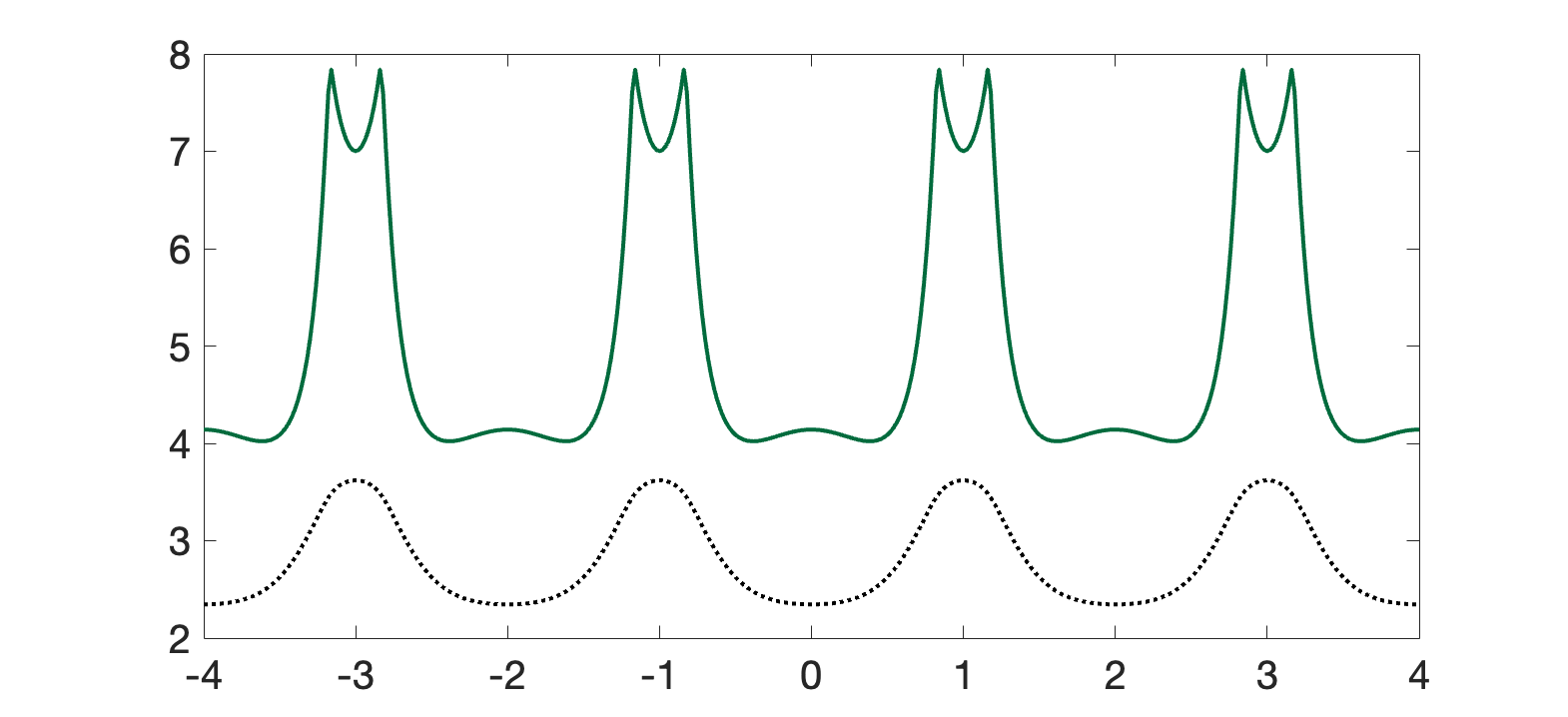}
        \put(0,30){\textcolor{green!50!black}{$R$}, $T$}
        \put(50,-2){$x$}
	  \end{overpic}
    \vspace{.5cm}
    \caption{Double-peaked stationary pattern in 1D (green solid line: $R$-profile (roots biomass), black dotted line: $T$-profile (toxicity concentration) obtained by continuation in \texttt{pde2path} and numerical simulations of the cross-diffusion system for plant-biomass and toxicity studied in~\cite{GIS_2026}.}
    \label{fig:numsim_horns}
\end{figure}

The goal of this paper is to use GSPT techniques to study the shape and dynamics of far-from-equilibrium patterns in models containing cross-diffusion, within the theoretical framework developed in \cite{DoelmanVeerman.2015}. In particular, we focus on the double-peaked stationary patterns in the 1D case observed numerically in \cite{GIS_2026}, reproduced here in Fig.~\ref{fig:numsim_horns}. We first construct the far-from-equilibrium patterns reported in \cite{GIS_2026}, obtained both by direct numerical simulations and continuation with \texttt{pde2path}, for the particular case $\gamma=s=0$ with linear (saturated) transition rates; see Sections~\ref{sec:model_specific} and \ref{sec:patterns_specific}. This detailed analysis of the original plant-toxicity model serves both to illustrate our analytical approach and to highlight the main arguments in a concrete setting.

This naturally leads to the following ecologically relevant question: are the double-peaked patterns observed in \cite{GIS_2026} merely a consequence of the particular choice of transition rates, or do they persist under more general (albeit mathematically different) transition functions?

We then address this question by extending the analysis to a general class of cross-diffusion plant-toxicity models; see Sections~\ref{sec:model_class} and \ref{sec:patterns_general}. In particular, we identify classes of transition rates for the fast-reaction system that either allow or preclude the existence of these front-type far-from-equilibrium patterns. Our main findings are summarised in Results~\ref{res:patterns_general_necessaryconditions} and \ref{res:patterns_generalmodel}.

Unlike the detailed treatment of the original model, our analysis of the general model class \eqref{eq:model_class} focuses on the key structural properties of the transition rates $f$ and $g$ that determine whether the analytical approach developed in Section~\ref{sec:patterns_specific} can be carried over. In particular, the geometry of the critical manifold (see, e.g., Eq.~\eqref{eq:C0} and Fig.~\ref{fig:CM_cases}) emerges as the fundamental ingredient governing the existence of these patterns.

\section{The plant--autotoxicity model}\label{sec:model} 
In this section, we introduce the models that will be analysed throughout this work. We begin with the cross-diffusion model for plant biomass and autotoxicity proposed in~\cite{GIS_2026}, which represents, to the best of our knowledge, the first formulation incorporating cross-diffusion effects in this context. We then turn to a more general framework proposed in~\cite{Morgan_etal.2025}, in which the nonlinear interaction terms are not specified explicitly, but are instead defined through a set of structural assumptions. Both models will be studied in the subsequent sections using GSPT techniques.

\subsection{Specific plant-toxicity model} \label{sec:model_specific} 
In \cite{GIS_2026}, the authors exploited the fast-reaction limit to derive a cross-diffusion system describing the dynamics of plant biomass and autotoxicity concentration, with propagation reduction induced by the toxicity.

We consider a bounded domain $\Omega \subset \mathbb{R}^n$ and denote by $R=R(x,t)$ and $T=T(x,t), \, x\in \Omega, \, t\geq 0$, the plant biomass and autotoxicity concentration, respectively. The cross-diffusion model writes
\begin{subequations}\label{eq:model_specific}
\begin{align}
 \partial_t R - \Delta \left[(d_R - \alpha\,\phi(T))R\right] &= (g - \gamma\,\phi(T))\,R\left(1-\frac{R}{\hat{R}}\right) - (d+s\,\phi(T))\,R, \label{eq:model_specific_R}\\
 \partial_t T - d_T \Delta T &= c(d+s\,\phi(T))\,R-k\,T\label{eq:model_specific_T}.
\end{align}
\end{subequations}
Here, $d_R$ and $d_T$ denote the diffusion coefficients of $R$ and $T$, respectively, while $\alpha$ is the cross-diffusion coefficient and $\phi(T)$ is the cross-diffusion function. It is assumed that $d_R>\alpha$ and the function $\phi(T)$ is determined by the transition rates of the fast-reaction system~\cite{GIS_2026} and writes
\begin{equation}
\phi(T) = \left\{\begin{array}{rcl} T/T_c, & \text{if} & T<T_c \ , \\ 1, & \text{if} & T>T_c\ .\end{array}\right.  \label{eq:pwphi}    
\end{equation}
In the reaction part, root dynamics is described with a logistic growth term and a mortality term, both toxicity-dependent. Parameter $g$ denotes the roots' maximum growth rate, the parameter $\gamma$ is the growth-inhibition rate induced by toxicity, $\hat{R}$ is the carrying capacity of the (logistic) plant growth, $d$ is the natural death rate, and $s$ is the extra-mortality rate induced by toxicity. In the toxicity equation, the first term describes the toxicity production due to decomposed biomass with a conversion factor $c$, while the second term accounts for toxicity dilution with a constant rate $k$.

While the model without cross-diffusion terms (namely $\alpha=0$) does not present Turing instability, it is possible to observe stable Turing patterns even when $\gamma = s = 0$, namely driven by cross-diffusion only \cite{GIS_2026}.


\subsection{General plant-toxicity model class}\label{sec:model_class}
A natural generalisation of the former model was introduced and studied in~\cite{Morgan_etal.2025}. In this framework, the specific functional forms of the nonlinearities are not fixed a priori, but are instead characterised by a set of qualitative assumptions. Also in this model, we denote by $R=R(x,t)$ and $T=T(x,t), \, x\in \Omega, \, t\geq 0$ the plant biomass and autotoxicity concentration, respectively. Using a fast-reaction limit, the authors derived the following general model class describing the interaction between plant biomass and toxicity:
\begin{subequations}\label{eq:model_class}
    \begin{align}
    \partial_t R &= \Delta \left(d_{R_1}R - (d_{R_1}-d_{R_2}) \xi_2(R,T)\right) + h(R,T),\\
    \partial_t T &= d_T \Delta T + \mu\,\left(\eta_1\,\xi_1(R,T) + \eta_2\,\xi_2(R,T)\right) - k\,T,
    \end{align}
\end{subequations}
where $d_{R_1}>d_{R_2}$ and the functions $h, \; \xi_1$, and $\xi_2$ are given by
\begin{subequations}\label{eq:model_class_nonlinearities}
    \begin{align}
        h(R,T) &= \sum_{i=1}^2 \gamma_i\,\xi_i(R,T) \left(\hat{R}-R\right) - \eta_i\,\xi_i(R,T),\\
        \xi_1(R,T) &= \frac{g(T)}{f(T)+g(T)}\,R,\\
        \xi_2(R,T) &= \frac{f(T)}{f(T)+g(T)}\,R.
    \end{align}
\end{subequations}
In the fast-reaction model, the functions $f$ and $g$ model the (nonnegative) switching rates (also called transition rates) between healthy roots and exposed roots depending on the toxicity concentration $T$. Therefore, we take $f,\, g : [0,\infty) \to [0,\infty)$, and assume that $f$ is monotonically increasing, while $g$ is monotonically decreasing. Possible expressions for $f$ and $g$ are \cite{Morgan_etal.2025}:
\begin{itemize}
    \item Power functions:  \begin{subequations}\label{eq:fg_choices_powerfct}
        \begin{align}
            f(T) &= T^p,\\
            g(T) &= \frac{1}{(1+T)^q},
        \end{align}
    \end{subequations}
    with exponents $p,\, q > 0$.
    \item Holling-type II functions:
    \begin{subequations}\label{eq:fg_choices_HollingII}
        \begin{align}
         f(T) &= \frac{a_1 T + b_1}{c_1 T + d_1},\\
         g(T) &= \frac{a_2 T + b_2}{c_2 T + d_2},
        \end{align}
    \end{subequations}
    with $a_i,\, b_i,\, c_i,\, d_i \geq 0$ such that $a_1 d_1 - b_1 c_1 \geq 0$, to ensure that $f$ is increasing, and $a_2d_2-b_2 c_2 \leq 0$, to ensure that $g$ is decreasing. Without loss of generality, we assume $d_i>0$. Hence, we can write
    \begin{subequations}\label{eq:fg_choices_HollingII_rescaled}
        \begin{align}
         f(T) &= f_0 + \frac{d_f T}{c_f T+1},\\
         g(T) &= g_0 - \frac{d_g T}{c_g T+1},
        \end{align}
    \end{subequations}
    with $f_0,\, g_0, \, d_f, \, d_g, \, c_f \geq 0$ and $c_g > 0$.
    \item Saturation functions:
    \begin{subequations}\label{eq:fg_choices_sattrans}
        \begin{align}
            f(T) &= T,\\
            g(T) &= \left\{ \begin{array}{rcl} \hat{T} - T & \text{if} & T < \hat{T},\\
            0 & \text{if} &T > \hat{T}
            \end{array} \right.~\, ,
        \end{align}
    \end{subequations}
    where $\hat{T}$ is a parameter representing a threshold value for the toxicity concentration. Possibly, $g$ should be smoothened to a $C^2$ approximation.
\end{itemize}
In general, we assume that both $f$ and $g$ are piecewise $C^1$.

\section{Periodic multiscale patterns in the specific model}\label{sec:patterns_specific}

To recast the model in the GSPT framework, we first rescale model~\eqref{eq:model_specific} by considering the adimensional quantities
\begin{equation}
    \tilde{R} = \frac{R}{\hat{R}}, \quad \tilde{T} = \frac{T}{T_c}, \quad \tilde{t} = k t, \quad \tilde{x} = x\sqrt{\frac{k}{d_R}},
\end{equation}
and by redefining the parameters
\begin{equation}
    \tilde{g} = \frac{g}{k}, \quad \tilde{s} = \frac{s}{k}, \quad \tilde{d} = \frac{d}{k}, \quad \tilde{\gamma} = \frac{\gamma}{k}, \quad h = \frac{c \hat{R}}{T_c}, \quad \rho 
    = \frac{\alpha}{d_R},  \quad \eps^2 = \frac{d_T}{d_R}.
\end{equation}
We obtain the adimensionalised system 
(where we drop tildes to enhance readability)
\begin{subequations}\label{eq:model_specific_rescaled}
\begin{align}
 \partial_t R - \Delta \left[(1 - \rho\,\phi(T))R\right] &= (g - \gamma\,\phi(T))\,R\left(1-R\right) - (d+s\,\phi(T))\,R, \label{eq:model_specific_R_rescaled}\\
 \partial_t T - \eps^2 \Delta T &= h(d+s\,\phi(T))\,R-T\label{eq:model_specific_T_rescaled},
\end{align}
\end{subequations}
where the cross-diffusion function is given by
\begin{equation}\label{eq:def_phiT}
        \phi(T) = \left\{\begin{array}{rcl} T & \text{if} & T< 1,\\ 1 & \text{if} & T>1.\end{array}\right. 
\end{equation}
The cross-diffusion parameter $\rho$ satisfies $0<\rho<1$, commensurate with the modelling assumption of the fast-reaction system that the diffusion coefficient for healthy roots is larger than the diffusion coefficient for roots that are exposed to toxicity \cite{GIS_2026}; for the general model \eqref{eq:model_class}, this is equivalent to the assumption $d_{R_1}>d_{R_2}$. All other parameters in \eqref{eq:model_specific_rescaled} are positive without a priori bounds.


We consider model \eqref{eq:model_specific_rescaled} in one spatial dimension; furthermore, we take $x \in \mathbb{R}$ unbounded. Our aim is to find stationary solutions to \eqref{eq:model_specific_rescaled} whose amplitude is bounded in $x$, such as fronts, pulses, or periodic concatenations of these. 
We consider $(P(x),T(x))$ and define $\frac{\text{d} P}{\text{d}x} := S$, $\eps \frac{\text{d} T}{\text{d} x} := U$. Using the shorthand notation $(\cdot)' = \frac{\text{d}}{\text{d} x}$, we can represent these stationary, bounded solutions as bounded orbits of the following dynamical system:
\begin{subequations}\label{eq:model_specific_dynsys}
    \begin{align}
        P' &= S,\\
        S' &= \left(d+s\,\phi(T)\right)\,R - \left(g - \gamma\,\phi(T)\right)\,R\left(1-R\right),\\
        \eps T' &= U,\\
        \eps U' &= T - h\left(d + s\,\phi(T)\right)\,R,\\
        \left(1 - \rho\,\phi(T)\right)R &= P.~\label{eq:model_specific_dynsys_constraint}
    \end{align}
\end{subequations}
The algebraic equation \eqref{eq:model_specific_dynsys_constraint} constrains the dynamics of \eqref{eq:model_specific_dynsys} to a 4-dimensional submanifold of $(P,S,T,U,R)$ phase space. The most straightforward way to proceed is to express $R$ in terms of $P$ and $T$:
\begin{equation}\label{eq:R_in_PT}
    R = \frac{P}{1 - \rho\,\phi(T)}.
\end{equation}
Note that the singularity that occurs when $\rho \phi(T) = 1$ signifies the transition from ``normal'' diffusion to anti-diffusion in \eqref{eq:model_specific_R}. Using the substitution \eqref{eq:R_in_PT} in \eqref{eq:model_specific_dynsys} yields:
\begin{subequations}\label{eq:model_specific_dynsys_inP}
    \begin{align}
        P' &= S,\\
        S' &= \frac{P}{1 - \rho\,\phi(T)}\left[d+s\,\phi(T) - \left(g - \gamma\,\phi(T)\right)\left(1-\frac{P}{1 - \rho\,\phi(T)}\right)\right],\\
        \eps T' &= U,\\
        \eps U' &= T - h\,P\frac{d + s\,\phi(T)}{1 - \rho\,\phi(T)},
    \end{align}
\end{subequations}
In order to facilitate the presentation of the upcoming analysis, we focus on the simplified scenario $\gamma=s=0$. We will mention where and how the more general choice $\gamma \neq 0$ and $s \neq 0$ would substantially change the results we present here.
Furthermore, we include the additional assumption used in \cite{GIS_2026} that $h=d^{-1}$ (imposing a specific expression on $T_c$ in terms of the other parameters). In this case, system \eqref{eq:model_specific_dynsys_inP} becomes
\begin{subequations}\label{eq:gammas0_slow}
    \begin{align}
        P' &= S,\\
        S' &= \frac{P}{1 - \rho\,\phi(T)}\left[d - g \left(1-\frac{P}{1 - \rho\,\phi(T)}\right)\right],\\
        \eps T' &= U,\\
        \eps U' &= T - \frac{P}{1 - \rho\,\phi(T)}.
    \end{align}
\end{subequations}
In the following subsections, we will analyse system \eqref{eq:gammas0_slow} in the regime where the toxicity diffusion rate $d_T$ is small in comparison to the biomass diffusion rate $d_R$, corresponding to $0< \eps \ll 1$. We will show the existence of far-from-equilibrium, multiple-scale periodic patterns using methods from Geometric Singular Perturbation Theory.

\subsection{Critical manifold}

The singular limit $\eps \to 0$ in \eqref{eq:gammas0_slow} yields two algebraic equations that determine the critical manifold, which is given by
\begin{equation}\label{eq:C0}
    \mathcal{C}_0 := \left\{U = 0\;, T = \frac{P}{1 - \rho\,\phi(T)}\right\} = \left\{U = 0\;, P = T(1 - \rho\,\phi(T))\right\}.
\end{equation}
The critical manifold $\mathcal{C}_0$ is naturally divided into different pieces due to the piecewise smooth nature of $\phi(T)$. Furthermore, the geometry of $\mathcal{C}_0$ can be further specified based on the value of $\rho$. We define
\begin{equation}\label{eq:C0_det}
\begin{aligned}
    \mathcal{C}_0^l &= \left\{U = 0\;, P = T(1 - \rho\,T)\;, 0 \leq T < \text{min}\left(\frac{1}{2 \rho},1\right)\right\}, \\
    \mathcal{C}_0^m &= \left\{U = 0\;, P = T(1 - \rho\,T)\;, \text{min}\left(\frac{1}{2 \rho},1\right) \leq T < 1\right\}, \\
    \mathcal{C}_0^r &= \left\{U = 0\;, P = T(1 - \rho)\;, T \geq 1 \right\},
\end{aligned}
\end{equation}
so that $\mathcal{C}_0 = \mathcal{C}_0^l \cup \mathcal{C}_0^m \cup \mathcal{C}_0^r$. This division of $\mathcal{C}_0$ leads to a natural case distinction based on the value of $\rho$, see also Fig.~\ref{fig:CM_cases}.

\begin{figure}
    \centering
    \begin{overpic}[width=0.7\linewidth]{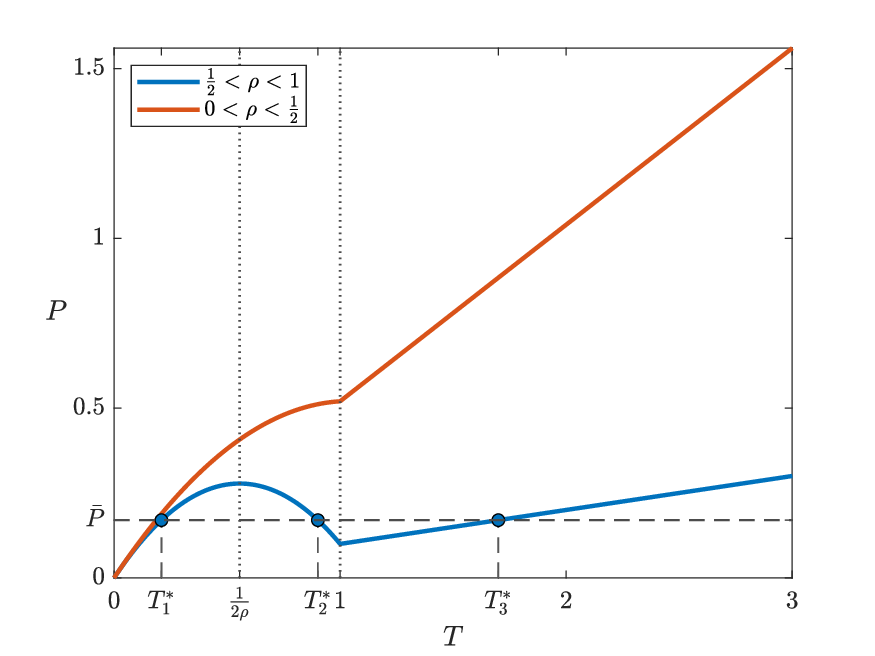}
        \put(17,20){\textcolor{matb}{\footnotesize{$\mathcal{C}_0^{l}$}}}
        \put(31,20){\textcolor{matb}{\footnotesize{$\mathcal{C}_0^{m}$}}}
        \put(70,20){\textcolor{matb}{\footnotesize{$\mathcal{C}_0^{r}$}}}
    \end{overpic}
    \caption{The projection of the critical manifold $\mathcal{C}_0$ \eqref{eq:C0} onto $(T,P)$-space for $0<\rho<1/2$ (orange) and for $1/2 < \rho < 1$ (blue). In the latter case, there is a bistability regime for the layer problem \eqref{eq:gammas0_layp}. The vertical dotted lines indicate the location of $T=1/(2\rho)$ and $T=1$, respectively, dividing the critical manifold in the three pieces $\mathcal{C}_0^{l}$ (left), $\mathcal{C}_0^{m}$ (middle), and $\mathcal{C}_0^{r}$ (right). The horizontal dashed line at a fixed value of $P=\bar{P}$ shows the existence of three intersection points (dots) for $T=T_i^\ast$, $i = 1,2,3$, as defined in Eq.~\eqref{eq:equilCM}.}
    \label{fig:CM_cases}
\end{figure}

\begin{figure}
    \centering
    \includegraphics[width=0.7\linewidth]{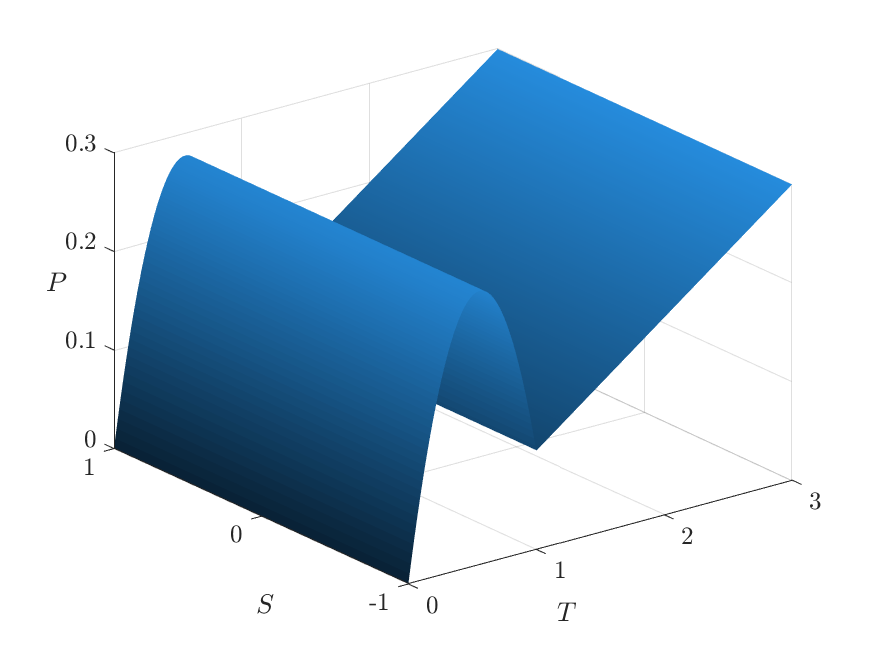}
    \caption{Critical manifold $\mathcal{C}_0$ \eqref{eq:C0} in $(T,S,P)$-space for $1/2 < \rho < 1$.}
    \label{fig:CM_3dim}
\end{figure}

We observe that $P$ is always an increasing function of $T$ on $\mathcal{C}_0^r$ and on $\mathcal{C}_0^l$. When $0 < \rho < 1/2$, we see that $\mathcal{C}_0^m = \emptyset$, so $P$ is increasing in $T$ on the entirety of $\mathcal{C}_0$.
On the other hand, for $1/2 < \rho < 1$ we have that $P$ is a decreasing function of $T$ on $\mathcal{C}_0^m \neq \emptyset$, and the manifolds $\mathcal{C}_0^l$, $\mathcal{C}_0^m$ meet at the fold line $\left\{(P = 1/(4\rho), T = 1/(2\rho), U = 0\right\}$ determined by the local maximum of the graph of $P = T(1-\rho T)$. The segments $\mathcal{C}_0^l$ and $\mathcal{C}_0^r$ (for $0 < \rho < 1/2$) or $\mathcal{C}_0^m$ and $\mathcal{C}_0^r$ (for $1/2 < \rho < 1$) meet at the line $\left\{ P = 1-\rho, T = 1, U = 0\right\}$, determined by the non-smooth transition of $\phi(T)$ at $T=1$ \eqref{eq:def_phiT}, see again Fig.~\ref{fig:CM_cases}. This non-smoothness of $\mathcal{C}_0$ will require further care in the following analysis.\\
As an aside, we note that the qualitative geometric properties of $\mathcal{C}_0$ do not change when $\gamma \neq 0$, $s \neq 0$  and $h \neq d^{-1}$ are taken in \eqref{eq:model_specific_dynsys_inP}; only the position of the fold $T = 1/(2\rho)$ will now depend on the fraction $s/d$.

\subsection{Layer problem}
The fast formulation of system \eqref{eq:model_specific_dynsys_inP} is found by introducing the rescaled spatial variable $\xi = x/\eps$, yielding
\begin{subequations}\label{eq:gammas0_fast}
    \begin{align}
        \dot{P} &= \eps\,S,\\
        \dot{S} &= \frac{\eps P}{1 - \rho\,\phi(T)}\left[d- g\left(1-\frac{P}{1 - \rho\,\phi(T)}\right)\right],\\
        \dot{T} &= U,\\
        \dot{U} &= T - \frac{P}{1 - \rho\,\phi(T)}.
    \end{align}
\end{subequations}
Subsequently taking the limit $\eps \to 0$, we obtain the layer problem
\begin{subequations}\label{eq:gammas0_layp}
    \begin{align}
        \dot{T} &= U,\\
        \dot{U} &= T - \frac{P}{1 - \rho\,\phi(T)},
    \end{align}
\end{subequations}
for constant values of the variables $P$, $S$.~\\
The equilibria $(T^\ast,U^\ast)$ of the layer problem \eqref{eq:gammas0_layp} are determined by the same equations that define the critical manifold $\mathcal{C}_0$ \eqref{eq:C0}. Hence, for fixed values $P = \bar{P}$, $S = \bar{S}$, we find $U^\ast=0$
and up to three possible $T^\ast$ values, namely
\begin{equation} \label{eq:equilCM}
    T_1^\ast = \frac{1-\sqrt{1-4\rho \bar{P}}}{2 \rho}, \qquad T_2^\ast = \frac{1+\sqrt{1-4\rho \bar{P}}}{2 \rho}, \qquad T_3^\ast = \frac{\bar{P}}{1-\rho}.
\end{equation}
In particular, if $0 < \rho < 1/2$ the only admissible value is $T_3^\ast$, whereas for $1/2 < \rho < 1$ we have the following scenarios:
\begin{equation}
\begin{cases}
    T_1^\ast & \text{ if } 0 < \bar{P} < 1-\rho, \\
    T_1^\ast, T_2^\ast & \text{ if } \bar{P} = 1-\rho, \\
    T_1^\ast, T_2^\ast, T_3^\ast & \text{ if } 1-\rho < \bar{P} < 1/(4\rho), \\
    T_2^\ast, T_3^\ast & \text{ if } \bar{P} = 1/(4\rho), \\
    T_3^\ast & \text{ if } \bar{P} > 1/(4\rho).
\end{cases}
\end{equation}
Note that these scenarios directly follow from intersecting the graph of $P = T(1-\rho \phi(T))$ with the horizontal line $P = \bar{P}$, cf. Fig.~\ref{fig:CM_cases}.

\subsubsection{Heteroclinic connections in the layer problem}
In phase space, the type of patterns observed in numerical simulations (see Fig.~\ref{fig:numsim_horns}) corresponds to periodic orbits consisting of slow orbits on $\mathcal{C}_0^l$ and $\mathcal{C}_0^r$ matched with fast (heteroclinic) connections between these manifolds. Therefore, in our subsequent analysis we focus on the case $1/2 < \rho < 1$ and in the layer problem \eqref{eq:gammas0_fast} we assume $1-\rho < \bar{P} < 1/(4\rho)$ in order to construct double heteroclinic connections between $(T_1^\ast,0)$ and $(T_3^\ast,0)$. This in turn implies that the range of admissible $T$-values in the analysis of the layer problem is $$T \in \left(\frac{1-\rho}{\rho}, \frac{1}{4\rho(1-\rho)}\right),$$ which by construction contains the equilibria $T_{1,2,3}^\ast$.~\\
Under these assumptions, we consistently have that the manifolds $\mathcal{C}_0^l$ and $\mathcal{C}_0^r$ are normally hyperbolic \cite{Kuehn.2015}: at the equilibria $(T_1^\ast, 0 )$ respectively $(T_3^\ast, 0 )$, the Jacobian matrix associated to eq.~\eqref{eq:gammas0_layp} admits the eigenvalues
\[
\lambda_1^\pm = \pm \sqrt{2}\frac{\sqrt{1-4 \rho \bar{P} +\sqrt{1-4 \rho \bar{P} }}}{1+\sqrt{1-4 \rho \bar{P} }},
\qquad \lambda_3^\pm = \pm 1,
\]\\
which are real and nonzero for $1-\rho < \bar{P} < 1/(4\rho)$.

We now prove the following result.
\begin{prop} \label{prop:hetconn}
    For any $1/2 < \rho < 1$ and any $S=\bar{S} \in \mathbb{R}$, there exists a unique $P^\ast(\rho)$ with $1-\rho < P^\ast(\rho) < 1/(4\rho)$ such that system \eqref{eq:gammas0_layp} admits a double heteroclinic connection $\varphi_\mathrm{h}^\pm(\xi) = \left(T_\mathrm{h}(\xi),\pm U_\mathrm{h}(\xi)\right)$ between $Q_1=(T_1^\ast, 0 )$ and $Q_3=(T_3^\ast, 0 )$. In other words, for $P = P^\ast(\rho)$, the intersection $\mathcal{W}^{u}(Q_1) \cap \mathcal{W}^{s}(Q_3)$ (resp.~$\mathcal{W}^{s}(Q_1) \cap \mathcal{W}^{u}(Q_3)$) is transversal.
\end{prop}
\begin{proof}
The vector field of the layer problem \eqref{eq:gammas0_layp} is non-smooth for $T_1^\ast < T < T_3^\ast$, since $1 \in (T_1^\ast, T_3^\ast)$. In particular, at $T=1$, the vector field is continuous but not continuously differentiable. To determine the shape of the stable and unstable manifolds of $Q_1$ and $Q_3$, we study the vector field on both sides of the vertical line $T=1$.\\
The orbits of system \eqref{eq:gammas0_layp} passing through points on $\mathcal{C}_0^l$ are obtained by solving
\begin{equation}
\begin{aligned}
    \frac{\text{d}U}{\text{d}T} &= \frac{1}{U} \left(T - \frac{P}{1 - \rho\,T} \right), \\
    U(T_1^\ast) &= 0,
\end{aligned}
\end{equation}
by separation of variables, which leads to
\begin{equation}
    U_1^\pm (T) = 
    \pm \sqrt{T^2 - (T_1^\ast)^2 + 2 \frac{P}{\rho} \log \frac{1-\rho T}{1-\rho T_1^\ast}},\quad \frac{1-\rho}{\rho} < T < 1.
\end{equation}
In particular, we have that the stable (resp.~unstable) manifold $\mathcal{W}^{s}$ (resp.~$\mathcal{W}^{u}$) of a generic point $Q_1 \in \mathcal{C}_0^l$ to the left of the vertical line $T=1$ is given by
\begin{equation}
\begin{aligned}
    \mathcal{W}^{s} (Q_1) &= \left\{ (P, S, T, U) \, : \, P=\bar{P}, \, S=\bar{S},\, \frac{1-\rho}{\rho}<T<1, \,  U=U_1^-(T)  \right\}, \\
    \mathcal{W}^{u} (Q_1) &= \left\{ (P, S, T, U) \, : \, P=\bar{P}, \, S=\bar{S},\, \frac{1-\rho}{\rho}<T<1, \,  U=U_1^+(T)  \right\}.
\end{aligned}
\end{equation}
Analogously, the orbits of system \eqref{eq:gammas0_layp} passing through points on $\mathcal{C}_0^r$ are obtained by solving
\begin{equation} \label{eq:UTsys_3}
\begin{aligned}
    \frac{\text{d}U}{\text{d}T} &= \frac{1}{U} \left(T - \frac{P}{1 - \rho} \right), \\
    U(T_3^\ast) &= 0,
\end{aligned}
\end{equation}
yielding
\begin{equation}
    U_3^\pm (T) = \pm \left(\frac{P}{1-\rho } - T \right),\quad 1 < T < \frac{1}{4 \rho  (1-\rho )}.
\end{equation}
In this case, because of the linearity of system \eqref{eq:UTsys_3} when applying the method of separation of variables, stable and unstable manifolds of $Q_3 \in \mathcal{C}_0^r$ coincide on the interval $T \in \left(1, 1/(4 \rho  (1-\rho ))\right)$ with the corresponding stable and unstable eigenspaces $E^{s,u}(Q_3)$, namely
\begin{equation}
\begin{aligned}
    \mathcal{W}^{s} (Q_3) &= \left\{ (P, S, T, U) \, : \, P=\bar{P}, \, S=\bar{S},\, 1<T<\frac{1}{4 \rho  (1-\rho )}, \,  U=U_3^-(T)  \right\} = E^{s} (Q_3), \\
    \mathcal{W}^{u} (Q_3) &= \left\{ (P, S, T, U) \, : \, P=\bar{P}, \, S=\bar{S},\, 1<T<\frac{1}{4 \rho  (1-\rho )}, \,  U=U_3^+(T)  \right\} = E^{u} (Q_3).
\end{aligned}
\end{equation}
Our goal is now to prove that the intersection $\mathcal{W}^{s}(Q_1) \cap \mathcal{W}^{u}(Q_3)$ (resp.~$\mathcal{W}^{u}(Q_1) \cap \mathcal{W}^{s}(Q_3)$) is transversal. Thanks to the symmetry of the functions $U_{1,3}^\pm$ involved, it suffices to prove this result for one intersection, and the other one will follow directly. Therefore, we focus on $\mathcal{W}^{u}(Q_1) \cap \mathcal{W}^{s}(Q_3)$.\\
We check the intersection of these manifolds at $T=1$. 
That is, our aim is to show that there exists a unique solution $\bar{P}=P^\ast(\rho)$ with $1-\rho < P^\ast(\rho) < 1/(4\rho)$ such that $U_1^+(1)=U_3^-(1)$, i.e.~$U_1^+(1)-U_3^-(1)=0$. 
In order to show the existence and uniqueness of $\bar{P}=P^\ast(\rho)$, we note the following:
\begin{itemize}
    \item $U_1^+(1)$ is a strictly monotonically decreasing function of $\bar{P}$, since
    \begin{displaymath}
        \frac{\text{d}}{\text{d} P}U_1^+(1) = \frac{1}{\rho U_1^+(1)} \log \frac{1-\rho}{1-\rho T_1^\ast} < 0 
    \end{displaymath}
    because $T_1^\ast<1$. Moreover, $U_1^+(1)$ assumes
    a positive value at $\bar{P}=1-\rho$ (namely, the lower extremum of the range of admissible $\bar{P}$-values) and admits a zero at a value of $\bar{P} < 1/(4\rho)$.
    In fact, when substituting $T_1^\ast$ as in Eq.~\eqref{eq:equilCM}, we have that
    \[
    U_1^+(1) = \frac{\sqrt{2 \rho (\bar{P} + \rho) +\sqrt{1-4 \rho \bar{P} }+4 \rho \bar{P}  \log \left(\dfrac{2(1- \rho) }{\sqrt{1-4 \rho \bar{P} }+1}\right)-1}}{\sqrt{2} \rho },
    \]
    which satisfies $U_1^+(1) = 0 $ when
    \[
    2 \rho (\bar{P} + \rho) +\sqrt{1-4 \rho \bar{P} }+4 \rho \bar{P}  \log \left(\frac{2(1- \rho) }{1+\sqrt{1-4 \rho \bar{P} }}\right)-1=0,
    \]
    or, equivalently, when
    \[
    u_1^+(\bar{P}):=e^{-(2 \rho  (\bar{P}+\rho )+\sqrt{1-4 \rho \bar{P} }-1)/4 \rho \bar{P} }-\frac{2 (1-\rho )}{1+\sqrt{1-4 \rho \bar{P} }} = 0.
    \]
    The function $u_1^+(\bar{P})$ is negative at $\bar{P}=1-\rho$ and positive at $\bar{P}=1/(4 \rho)$ for any $1/2 < \rho < 1$; therefore, by the intermediate value theorem, it must admit a zero within this $\bar{P}$-range. 
    \item $U_3^-(1)$ is a linear, strictly monotonically increasing function of $\bar{P}$, equal to $0$ at $\bar{P}=1-\rho$.
\end{itemize}
Therefore, for any $\rho \in \left(1/2, 1 \right)$ there exists a unique $\bar{P}=P^\ast(\rho)$ with $1-\rho < P^\ast(\rho) < 1/(4\rho)$ such that $(U_1^+(1)-U_3^-(1))|_{\bar{P}(\rho)} = 0$. The transversality of the intersection $\mathcal{W}^{s}(Q_1) \cap \mathcal{W}^{u}(Q_3)$ follows from the non-degeneracy of $\bar{P}(\rho)$, i.e. from the transversality of the intersection of $U_1^+(1)$ and $U_3^-(1)$ as graphs over $P$.
\end{proof}

The double heteroclinic connection for a fixed value of $\rho$ can be visualised in Fig.~\ref{fig:doublehet}.
    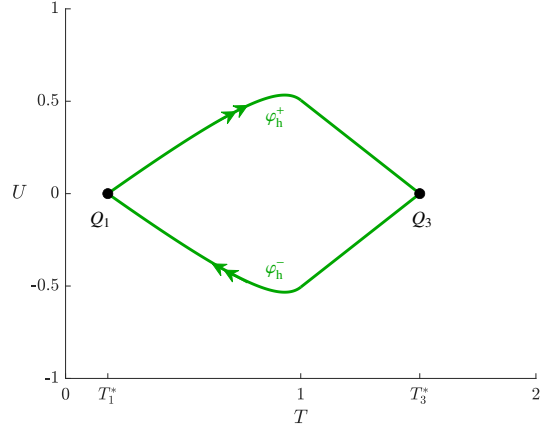
\begin{figure}[!ht]
        \centering
        \begin{overpic}[width=.6\linewidth]{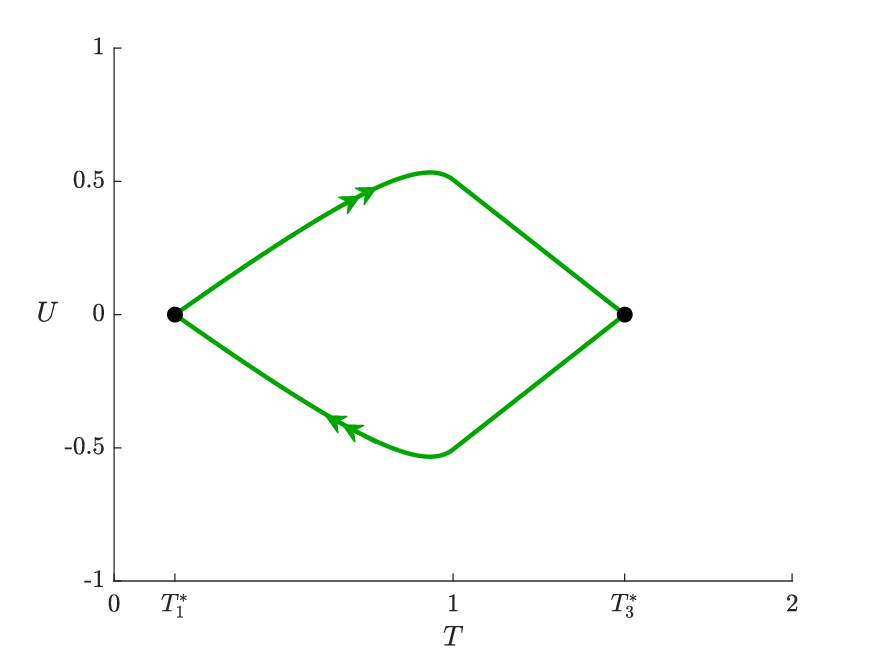}
            \put(17,34){\scriptsize{$Q_1$}}
            \put(70,34){\scriptsize{$Q_3$}}
            \put(46,51){\textcolor{darkg}{\scriptsize{$\varphi_\mathrm{h}^+$}}}
            \put(46,26){\textcolor{darkg}{\scriptsize{$\varphi_\mathrm{h}^-$}}}
        \end{overpic}\\
        \caption{Double heteroclinic connection $\varphi_\mathrm{h}^\pm$ for $\rho=0.9$ and $P^\ast(\rho) \approx 0.15$ in $(T,U)$-space between $Q_1$ and $Q_3$.
        }
        \label{fig:doublehet}
    \end{figure}




\subsection{Reduced problem}

The reduced problem is obtained by considering the singular limit $\eps \to 0$ in System \eqref{eq:gammas0_slow}, which leads to
\begin{subequations}\label{eq:gammas0_redpb}
    \begin{align}
        P' &= S,\\
        S' &= \frac{P}{1 - \rho\,\phi(T)}\left[d - g \left(1-\frac{P}{1 - \rho\,\phi(T)}\right)\right],
    \end{align}
\end{subequations}
together with the two algebraic equations that define the critical manifold $\mathcal{C}_0$ \eqref{eq:C0}.
In particular, we note that in virtue of eq.~\eqref{eq:C0} we have $P' = T'(1-\rho(\phi(T)-\phi'(T)T))$. On each of the submanifolds $\mathcal{C}_0^{l,m,r}$ \eqref{eq:C0_det}, we have that $1-\rho(\phi(T)-\phi'(T)T) \neq 0$ by construction, which allows us to rewrite System \eqref{eq:gammas0_redpb} on every individual submanifold $\mathcal{C}_0^{l,m,r}$ in terms of the variables $T$ and $S$ as
\begin{subequations}\label{eq:gammas0_redpb_new}
    \begin{align}
        T' &= \frac{S}{1-\rho \phi(T)-\rho \phi'(T)T},\\
        S' &= T\left( d-g(1-T)\right).
    \end{align}
\end{subequations}
In our subsequent analysis, we focus our attention on the reduced flow on $\mathcal{C}_0^l$ and $\mathcal{C}_0^r$, since these are the manifolds on which the slow parts of the sought-after singular periodic orbit will turn out to be located; this is a direct consequence of the fact that $\mathcal{C}_0^l$ and $\mathcal{C}_0^r$ are normally hyperbolic, whereas $\mathcal{C}_0^m$ is not.\\
In addition, to connect the analysis of the reduced flow to the scenario in the layer problem where heteroclinic connections exist (cf. Proposition \ref{prop:hetconn}), we will restrict the analysis of the reduced problem to the parameter regime $1/2 < \rho < 1$.\\

\subsubsection{The reduced problem on \texorpdfstring{$\mathcal{C}_0^l$}{C0l}} \label{sec:redC0l}
On $\mathcal{C}_0^l$, we have \mbox{$1-\rho \phi(T)-\rho \phi'(T)T = 1-2\rho T$} (cf.~\eqref{eq:C0_det}); from the definition of $\mathcal{C}_0^l$, we see that $1 - 2 \rho T > 0$ on $\mathcal{C}_0^l$. We can therefore apply a rescaling of the independent variable in \eqref{eq:gammas0_redpb_new} by multiplying both right-hand sides by $1-2\rho T$ to obtain
\begin{subequations}\label{eq:gammas0_redpb_new_l}
    \begin{align}
        T' &= S,\\
        S' &= T\left( d-g(1-T)\right)(1-2\rho T).
    \end{align}
\end{subequations}
System \eqref{eq:gammas0_redpb_new_l} has a Hamiltonian function $H_l$, that is given by
\begin{equation}\label{eq:C0l_Hamiltonian}
    H_l(T,S) = \frac{1}{2}S^2 + T^2 \left(\frac{g-d}{2} + \frac{1}{3}(2 d \rho -g(1+2\rho)) T + \frac{g \rho}{2} T^2\right).
\end{equation}
The system admits three equilibria $(\bar{T}_i,0)$, $i=1,2,3$, where
\begin{equation}\label{eq:equilibria_redpb}
\bar{T}_1 = 0, \qquad \bar{T}_2=\frac{g-d}{g}, \qquad \bar{T}_3 = \frac{1}{2 \rho}.
\end{equation}
We note that $\bar{T}_1$ and $\bar{T}_3$ are located at the boundaries of $\mathcal{C}_0^l$. This is not problematic for our analysis: we analyse the orbit structure of system \eqref{eq:gammas0_redpb_new_l} for $T \in \mathbb{R}$, and a posteriori restrict the flow to $0 < T <1/(2\rho)$. In addition, we note that the steady-state $(\bar{T}_2,0)$ lies on $\mathcal{C}_0^l$ if and only if $g>d$ and $\rho<g/(2(g-d))$. In line with the biological ranges of the model parameters used in \cite{GIS_2026}, we henceforth assume $g > 2 d$ and $\rho > g/(2(g-d))$, so that $\bar{T}_2 > 1/(2 \rho)$. This assumption is simplifying rather than restrictive; the analysis below can be repeated for the case that $(\bar{T}_2,0)$ does lie on $\mathcal{C}_0^l$. While this would induce a qualitative change in the orbit structure of the reduced problem \eqref{eq:gammas0_redpb_new_l} on $\mathcal{C}_0^l$, the existence of `connecting' slow orbit segments on $\mathcal{C}_0^l$ would not be altered, hence the statement of the upcoming Proposition \ref{prop:singsol} would still hold.\\ 
The Jacobian matrix associated to \eqref{eq:gammas0_redpb_new_l} is given by
\[
J_l = \begin{pmatrix}
    0 & 1 \\
    d(1-4\rho T)-g\left(1+6\rho T^2-2T(1+2\rho)\right) & 0
\end{pmatrix},
\]
and has vanishing trace;from the structure of the Hamiltonian function $H_l$ \eqref{eq:C0l_Hamiltonian} and sign of the determinant of $J_l$ evaluated at $\bar{T}_{1,3}$, we conclude that $(\bar{T}_1,0)$ is a center, whereas $(\bar{T}_3,0)$ is a saddle. 
The orbits of system \eqref{eq:gammas0_redpb_new_l} lie on the level sets of $H_l$ \eqref{eq:C0l_Hamiltonian}. First, we observe that $\left\{H_l(T,S)) = 0 \right\} = \left\{(\bar{T}_1,0)\right\}$ and that $(\bar{T}_1,0)$ is a global minimum of $H_l$. Second, the level set containing $(\bar{T}_3,0)$ is characterised by $H_l(\bar{T}_3,0) = (4 \rho(g-d)-g)/(96 \rho^3) > 0$. Hence, all level sets $\left\{H_l(T,S) = H_0\right\}$ with $0 < H_0 < H_l(\bar{T}_3,0)$ orthogonally intersect the $S$-axis at a point in between $(\bar{T}_1,0)$ and $(\bar{T}_3,0)$, see also Fig.~\ref{fig:red_pb_C0l}.\\
To prepare the statement of the upcoming Proposition \ref{prop:singsol}, we note that the value $P^\ast(\rho)$ as found in Proposition \ref{prop:hetconn} characterises the existence of a pair of heteroclinic connections in the layer problem. Since on $\mathcal{C}_0^l$ the relation $P = T(1-\rho T)$ is a bijection, we define $T^\ast_l(\rho) = T^\ast_1(P^\ast(\rho)) = (1-\sqrt{1-4 \rho P^\ast(\rho)}/(2 \rho)$, cf. \eqref{eq:equilCM}. For future reference, we define the so-called take-off and touchdown sets $\mathcal{T}_o^l, \mathcal{T}_d^l \subset \mathcal{C}_0^l$ as
\begin{equation} \label{eq:totdl}
\begin{aligned}
    \mathcal{T}_o^l &= \left\{ (P, S, T, U) \, : \, P=P^\ast(\rho), \, - S^\ast_l(\rho) < S < 0,\, T=T_l^\ast(\rho), \, U = 0  \right\}, \\
    \mathcal{T}_d^l &= \left\{ (P, S, T, U) \, : \, P=P^\ast(\rho), \, 0 < S < S^\ast_l(\rho),\, T=T_l^\ast(\rho), \, U = 0 \right\},
\end{aligned}
\end{equation}
where $S^\ast_l(\rho)$ is defined as the positive $S$-value for which $H_l(T_l^\ast(\rho),S^\ast_l(\rho)) = H_l(\bar{T}_3,0)$, i.e.
\begin{align}\label{eq:Sl}
    S^\ast_l(\rho) &= \sqrt{H_l(\bar{T}_3,0)- (T_l^\ast(\rho))^2\left(g-d + \frac{2}{3}(2 d \rho -g(1+2\rho)) T_l^\ast(\rho) + g \rho (T_l^\ast(\rho))^2\right)}\\
    &= \sqrt{\frac{1-4 \rho P^\ast(\rho)}{96 \rho^3}}\sqrt{12 \rho(1+P^\ast(\rho)) - 9 g - 12 d \rho + 8(g - \rho(g-d))\sqrt{1-4 \rho P^\ast(\rho)}}.~\nonumber
\end{align}



\begin{figure}[!ht]
    \centering
    \begin{overpic}[scale=0.6]{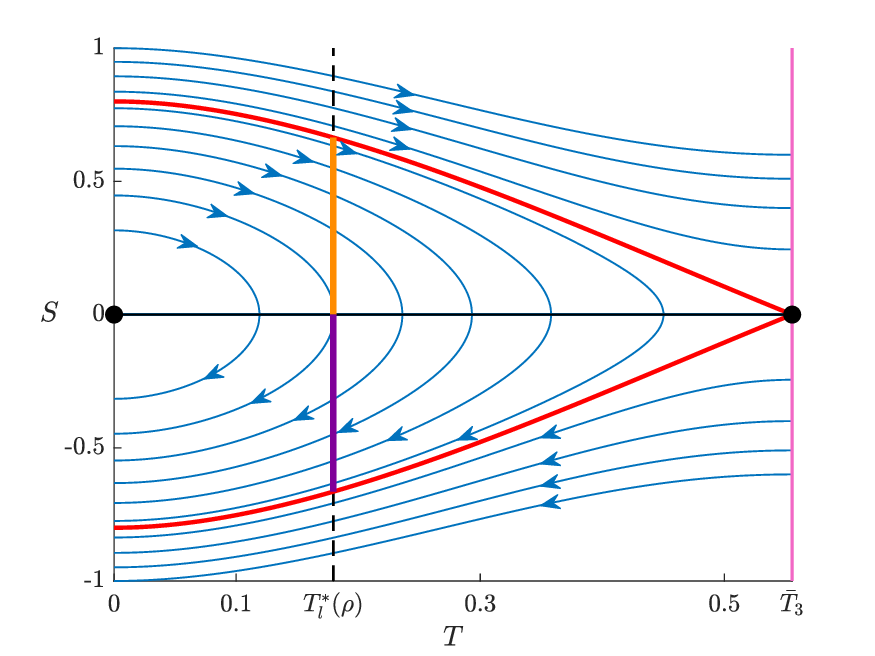}
    \end{overpic}
    \caption{Sketch of the dynamics for the reduced problem \eqref{eq:gammas0_redpb_new_l} on $\mathcal{C}_0^l$ in $(T,S)$-space with $\rho = 0.9$. The black, dashed line corresponds to $T=T_l^\ast(\rho)$; the magenta vertical line is located at $T=\bar{T}_3$. The red curves correspond to the level set $\{ H_l(T,S)=H_l(\bar{T}_3,0) \}$. The orange (resp.~purple) curve represents the touch-down (resp.~take-off) set $\mathcal{T}_d^l$ (resp.~$\mathcal{T}_o^l$), see eq.~\eqref{eq:totdl}. The black dots mark the equilibria at $(\bar{T}_1,0)$ and $(\bar{T}_3,0)$, whereas the arrows indicate the flow direction.}
    \label{fig:red_pb_C0l}
\end{figure}

\subsubsection{The reduced problem on \texorpdfstring{$\mathcal{C}_0^r$}{C0r}}
On $\mathcal{C}_0^r$ we have \mbox{$1-\rho \phi(T)-\rho \phi'(T)T = 1-\rho$} (cf.~\eqref{eq:C0_det}), which is again a strictly positive quantity. Applying a rescaling of the independent variable in \eqref{eq:gammas0_redpb_new} by multiplying both right-hand sides by $1-\rho$ leads to
\begin{subequations}\label{eq:gammas0_redpb_new_r}
    \begin{align}
        T' &= S,\\
        S' &= T\left( d-g(1-T)\right)(1-\rho).
    \end{align}
\end{subequations}
System \eqref{eq:gammas0_redpb_new_r} has a Hamiltonian function $H_r$, that is given by
\begin{equation}\label{eq:C0r_Hamiltonian}
    H_r(T,S) = \frac{1}{2}S^2 + (1-\rho)T^2 \left(\frac{g-d}{2} - \frac{g}{3}T \right).
\end{equation}
The equilibria admitted by this system are $(\bar{T}_1,0)$ and $(\bar{T}_2,0)$ \eqref{eq:equilibria_redpb}; however, as both $\bar{T}_1 < 1$ and $\bar{T}_2<1$, these equilibria do not lie on $\mathcal{C}_0^r$ \eqref{eq:C0_det}. That being said, we can use the nature of $(\bar{T}_1,0)$ and $(\bar{T}_2,0)$ to determine the orbit structure of system \eqref{eq:gammas0_redpb_new_r}, and a posteriori restrict to $T>1$. Based on the Jacobian matrix associated to \eqref{eq:gammas0_redpb_new_r}, given by
\begin{equation}
J_r = \begin{pmatrix}
    0 & 1 \\
    (1-\rho)\left(d-g(1-2T)\right) & 0
\end{pmatrix},
\end{equation}
and the structure of the Hamiltonian function $H_r$ \eqref{eq:C0r_Hamiltonian}, we conclude that $(\bar{T}_1,0)$ is a center, whereas $(\bar{T}_2,0)$ is a saddle. The orbits of system \eqref{eq:gammas0_redpb_new_r} lie on the level sets of $H_r$ \eqref{eq:C0r_Hamiltonian}, and the level set containing $(\bar{T}_2,0)$ is characterised by $H_r(\bar{T}_2,0) = (1-\rho)(g-d)^3/(6g^2)$. Hence, all level sets $\left\{H_r(T,S) = H_0\right\}$ with $H_0 < H_r(1,0) < H_r(\bar{T}_2,0)$ orthogonally intersect the $S$-axis to the right of $T=1$, see also Fig.~\ref{fig:red_pb_C0r}.\\
Again, with the heteroclinic connection of Proposition \ref{prop:hetconn} in mind, we define $T^\ast_r(\rho) = T^\ast_3(P^\ast(\rho)) = P^\ast(\rho)/(1-\rho)$, cf. \eqref{eq:equilCM}. For future reference, we define the take-off and touchdown sets $\mathcal{T}_o^r, \mathcal{T}_d^r \subset \mathcal{C}_0^r$ as
\begin{equation} \label{eq:totdr}
\begin{aligned}
    \mathcal{T}_o^r &= \left\{ (P, S, T, U) \, : \, P=\bar{P}(\rho), \, 0 < S < \bar{S}^\ast_r(\rho),\, T=T^\ast_r(\rho), \, U =0  \right\}, \\
    \mathcal{T}_d^r &= \left\{ (P, S, T, U) \, : \, P=\bar{P}(\rho), \,  - \bar{S}^\ast_r(\rho) < S < 0,\, T=T^\ast_r(\rho), \, U =0  \right\},
\end{aligned}
\end{equation}
where $S^\ast_r(\rho)$ is defined as the positive $S$-value for which $H_r(T^\ast_r(\rho),S^\ast_r(\rho)) = H_r(1,0)$, i.e.
\begin{align}\label{eq:Sr}
    S^\ast_r(\rho) &= \sqrt{H_r(1,0) - (1-\rho)(T^\ast_r(\rho))^2\left(g-d-\frac{2}{3}g T^\ast_r(\rho)\right)}\\
    &= \sqrt{\frac{(1-\rho-P^\ast(\rho))\left(g(1-\rho+2 P^\ast(\rho))(1-\rho-P^\ast(\rho)) - 3 d (1-\rho-P^\ast(\rho))(1-\rho)\right)}{3(1-\rho)^2}}.~\nonumber
\end{align}

\begin{rmk}
    We note that in general $\bar{S}_l \neq \bar{S}_r$, cf.~\eqref{eq:Sl} and \eqref{eq:Sr}; consequently, the bounds for the $S$-ranges in the take-off/touch-down sets $\mathcal{T}^l_{o,d}$ \eqref{eq:totdl} and $\mathcal{T}_{o,d}^r$ \eqref{eq:totdr} should be restricted to the minimum between these two values, in order to ensure that an admissible heteroclinic connection can be established between $\mathcal{C}_0^l$ and $\mathcal{C}_0^r$.
\end{rmk}

\begin{figure}[!ht]
    \centering
    \includegraphics[scale=0.6]{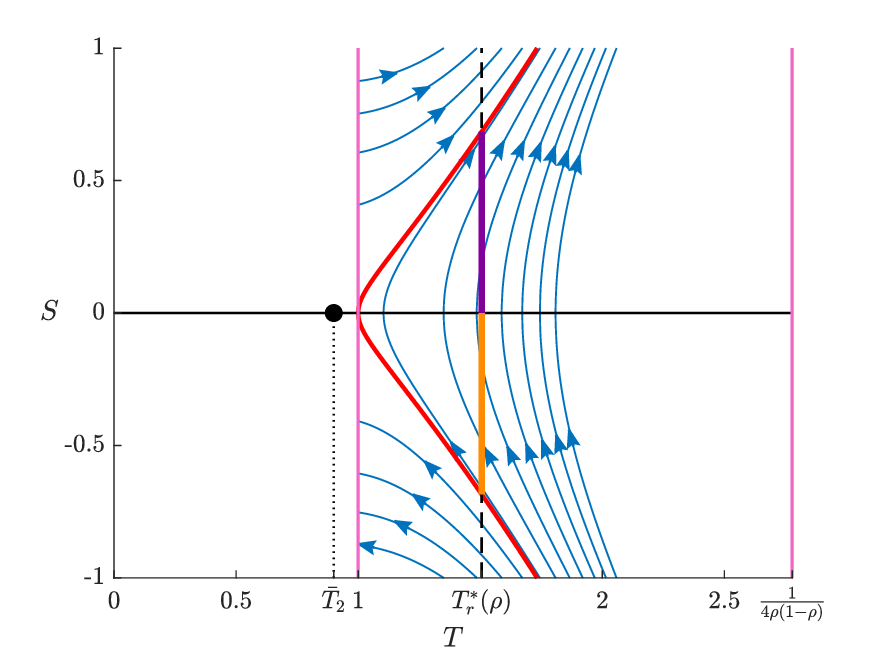}
    \caption{Sketch of the dynamics for the reduced problem \eqref{eq:gammas0_redpb_new_r} on $\mathcal{C}_0^r$ in $(T,S)$-space with $\rho = 0.9$. The black, dashed line corresponds to $T=T^\ast_r(\rho)$; the magenta vertical lines indicate $T=1$ and $T=1/(4\rho(1-\rho))$. The red curve corresponds to the level set $\left\{H_r(T,S) = H_r(1,0)\right\}$. The orange (resp.~purple) curve represents the touch-down (resp.~take-off) set $\mathcal{T}_d^r$ (resp.~$\mathcal{T}_o^r$),  see Eq.~\eqref{eq:totdr}. The arrows indicate the flow direction.}
    \label{fig:red_pb_C0r}
\end{figure}

\subsection{Singular periodic orbits}

The results presented in the previous subsections can be used to constructively prove the existence of far-from-equilibrium, multi-scale stationary periodic patterns in system \eqref{eq:model_specific}. We first provide the singular skeleton for these periodic patterns in Proposition \ref{prop:singsol}, which will allow us to prove the existence of periodic patterns for sufficiently small $0<\eps \ll 1$ in Theorem \ref{thm:po_existence}.

\begin{prop}[Singular skeleton] \label{prop:singsol}
    For $g > 2 d$, $g/(2(g-d)) < \rho < 1$ and $\varepsilon=0$, the (2,2)-fast-slow system \eqref{eq:gammas0_slow}, \eqref{eq:gammas0_fast} admits a one-parameter family of singular periodic orbits $\left\{ \Gamma_0^\mu \right\}_\mu$ consisting of precisely two fast and two slow subsystem trajectories with slow parts lying entirely in $\mathcal{C}_0^{l}$ and $\mathcal{C}_0^{r}$.
\end{prop}

\begin{proof}
Fixing a value $H_0 := \mu$ for the Hamiltonian $H_l$ \eqref{eq:C0l_Hamiltonian}
allows us to select one of the level curves of the reduced problem on $\mathcal{C}_0^l$. Choosing $0 < \mu < H_l(\bar{T}_3,0) = (4 \rho(g-d)-g)/(96 \rho^3)$ ensures that the level curve orthogonally intersects the $S$-axis at a point in between $(\bar{T}_1,0)$ and $(\bar{T}_3,0)$.\\
By Proposition \ref{prop:hetconn}, for any $1/2 < \rho < 1$ there is a unique value $P^\ast(\rho)$ such that a double heteroclinic connection exists between $(T,U) = (T^\ast_1,0)$ and $(T,U) = (T^\ast_3,0)$, which defines the take-off curve $\mathcal{T}_o^l$ \eqref{eq:totdl}. The level curve $\left\{H_l(T,S) = \mu\right\} \subset \mathcal{C}_0^l$ transversally intersects the vertical line segment $\mathcal{T}_o^l$ at a unique negative $S$-value $S=-\bar{S}^\mu$, where $\bar{S}^\mu>0$ is the unique positive solution to the equation $H_l(T^\ast_l,\bar{S}^\mu) = \mu$.\\
On $\mathcal{C}_0^r$, the horizontal line $\left\{S = \bar{S}^\mu\right\}$ transversally intersects the vertical line $\left\{T = T^\ast_r\right\}$. Since $\bar{S}^\mu$ is strictly monotonic in $\mu$ (cf.~\eqref{eq:C0l_Hamiltonian}), we can choose $\mu$ sufficiently small such that $H_r(T^\ast_r,\bar{S}^\mu) < H_r(1,0)$, thereby a posteriori restricting the range of $\mu$. This ensures that $(T^\ast_r,-\bar{S}^\mu) \in \mathcal{T}_d^r$. The level curve $\left\{H_r(T,S) = H_r(T^\ast_r,-\bar{S}^\mu)\right\}$ now uniquely determines the slow orbit segment $\psi_0^r$ on $\mathcal{C}_0^r$ that transversally intersects the take-off curve $\mathcal{T}_o^r$ at $(T,S) = (T^\ast_r,\bar{S}^\mu)$. This allows us to close the loop by concatenating the heteroclinic orbit $\varphi_\mathrm{h}^-$ from $(T,U) = (T^\ast_3,0)$ to $(T,U) = (T^\ast_1,0)$, ending up at the level set $\left\{H_l(T,S) = H_l(T^\ast_l,\bar{S}^\mu)\right\} = \left\{H_l(T,S) = H_l(T^\ast_l,-\bar{S}^\mu) = \mu\right\}$, uniquely determining the orbit segment $\psi_0^l$ on $\mathcal{C}_0^l$ that lies to the right of the vertical line $\left\{T = T^\ast_l\right\} \subset \mathcal{C}_0^l$.\\
The singular orbit $\Gamma_0^\mu$ is given by the concatenation $\varphi_\mathrm{h}^+ \cup \psi_0^r \cup \varphi_\mathrm{h}^- \cup \psi_0^l$.
\end{proof}

A sketch of a singular periodic orbit as constructed in Proposition~\ref{prop:singsol} for a fixed $\mu$ is shown in Figs.~\ref{fig:sketch_singsol}--\ref{fig:sketch_singsol_3D}.


\begin{figure}[!ht]
    \centering
    \begin{overpic}[scale=0.7]{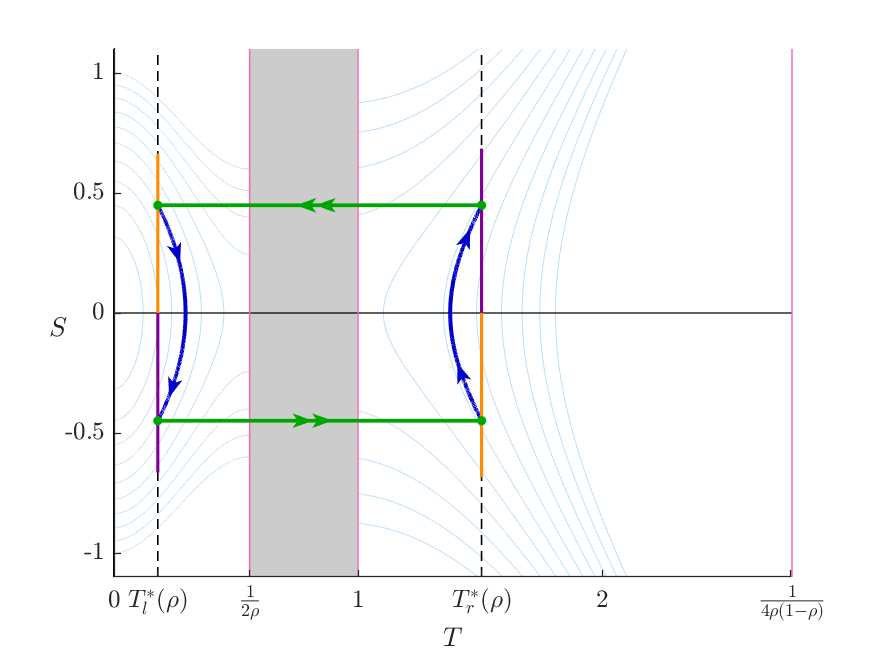}
    \put(20,62){\footnotesize{$\mathcal{C}_0^l$}}
    \put(33,62){\footnotesize{$\mathcal{C}_0^m$}}
    \put(78,62){\footnotesize{$\mathcal{C}_0^r$}}
    \put(22,41){\textcolor{darkb}{\footnotesize{$\psi_0^l$}}}
    \put(47.5,41){\textcolor{darkb}{\footnotesize{$\psi_0^r$}}}
    \put(34,29.5){\textcolor{darkg}{\footnotesize{$\varphi_\mathrm{h}^-$}}}
    \put(34,48.5){\textcolor{darkg}{\footnotesize{$\varphi_\mathrm{h}^+$}}}
    \end{overpic}
    \caption{Sketch of a singular periodic solution projected on $(T,S)$-space obtained by matching slow (blue) orbits $\psi_0^{l,r}$ on $\mathcal{C}_0^{l,r}$ with fast (green) heteroclinic ``jumps'' $\varphi_\mathrm{h}^\pm$ at $S=\pm \bar{S}^\mu$ with $\mu=0.2$. This essentially corresponds to an overlay of the regions visualised in Figs.~\ref{fig:red_pb_C0l}-\ref{fig:red_pb_C0r}. The gray area corresponds to $\mathcal{C}_0^m$, which is not involved in the construction.~\\
    }
    \label{fig:sketch_singsol}
\end{figure}

\begin{figure}[!ht]
    \begin{minipage}{.4\textwidth}
    \centering
    \begin{overpic}[scale=0.45]{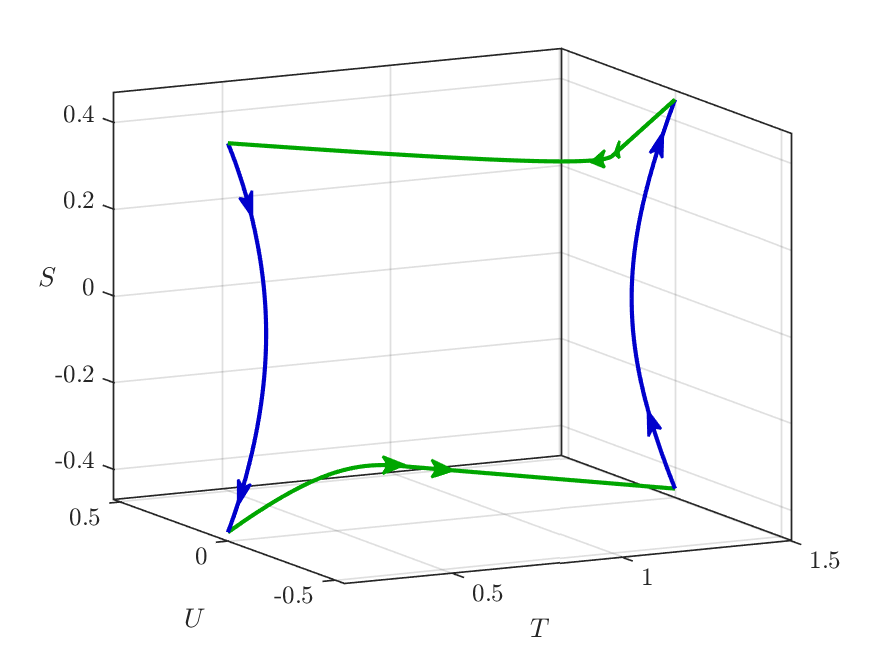}
    \put(24,38){\textcolor{darkb}{\footnotesize{$\psi_0^l$}}}
    \put(75,38){\textcolor{darkb}{\footnotesize{$\psi_0^r$}}}
    \put(45,25){\textcolor{darkg}{\footnotesize{$\varphi_\mathrm{h}^+$}}}
    \put(59,53){\textcolor{darkg}{\footnotesize{$\varphi_\mathrm{h}^-$}}}
    \end{overpic}\\
    (a)
    \end{minipage}
    \hspace{1cm}
    \begin{minipage}{.4\textwidth}
    \centering
    \begin{overpic}[scale=0.45]{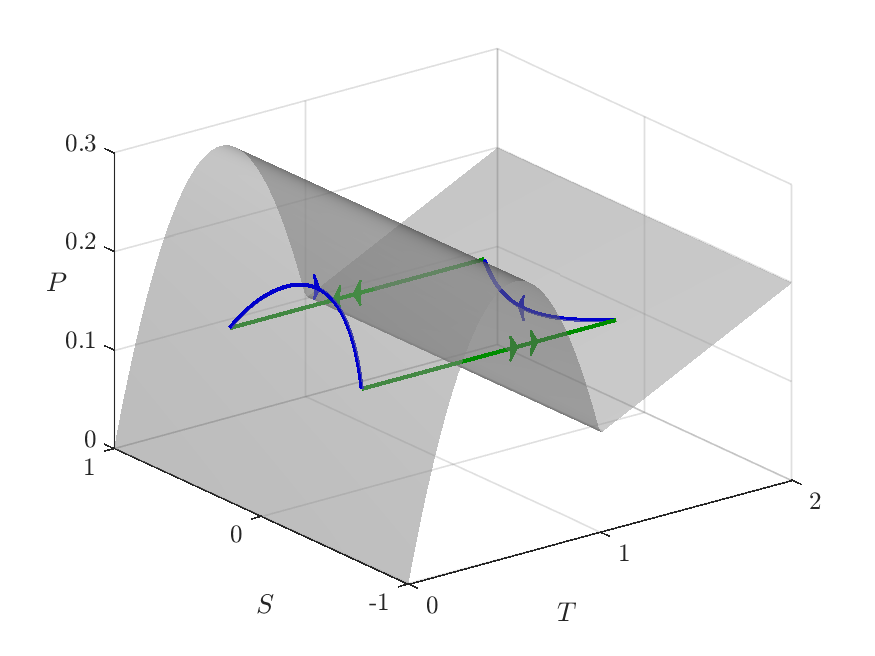}
    \put(28,45){\textcolor{darkb}{\footnotesize{$\psi_0^l$}}}
    \put(62,42){\textcolor{darkb}{\footnotesize{$\psi_0^r$}}}
    \put(31,36){\textcolor{darkg}{\footnotesize{$\varphi_\mathrm{h}^+$}}}
    \put(52,28){\textcolor{darkg}{\footnotesize{$\varphi_\mathrm{h}^-$}}}
    \end{overpic}\\
    (b)
    \end{minipage}     
    \caption{Sketch of a singular periodic solution projected in (a) $(T,U,S)$-space and (b) $(T,S,P)$-space obtained by matching slow (blue) orbits $\psi_0^{l,r}$ on $\mathcal{C}_0^{l,r}$ with fast (green) heteroclinic ``jumps'' $\varphi_\mathrm{h}^\pm$ at $S=\pm \bar{S}^\mu$ with $\mu=0.2$
    and $P=P^\ast(\rho)$. 
    The light gray surface in (b) represents the critical manifold $\mathcal{C}_0$ \eqref{eq:C0}.}
    \label{fig:sketch_singsol_3D}
\end{figure}
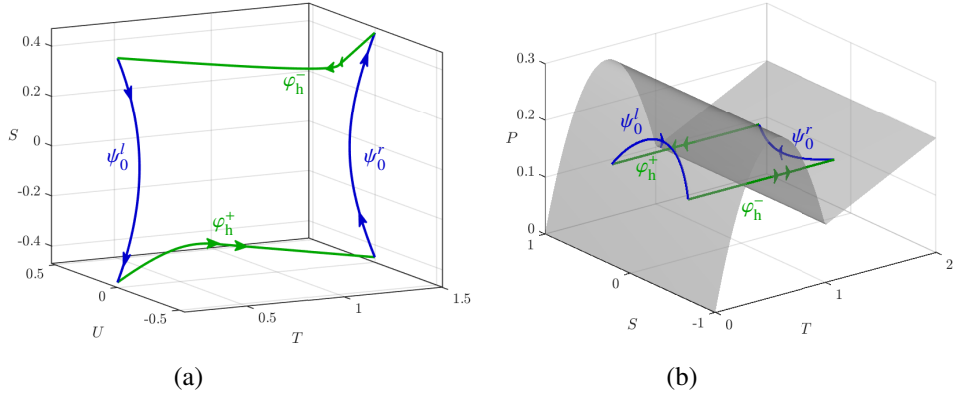


\begin{thm}[Persistence of periodic orbits]\label{thm:po_existence}
    Let $g>2d$ and $g/(2(g-d)) < \rho < 1$. There exists an $\eps_0>0$ such that for all $0<\eps<\eps_0$, system \eqref{eq:gammas0_slow} admits a one-parameter family $\left\{\Gamma_\eps^\mu\right\}_\mu$ of periodic orbits that are $\mathcal{O}(\eps)$-close to the singular skeleton family $\left\{\Gamma_0^\mu\right\}_\mu$ as defined in Proposition \ref{prop:singsol}.
\end{thm}

\begin{proof}
The proof is structured via the following steps:
\begin{enumerate}
    \item The two-dimensional \textbf{critical manifold} $\mathcal{C}_0$ \eqref{eq:C0} is composed of three separate two-dimen\-sional submanifolds: $\mathcal{C}_0^l$, $\mathcal{C}_0^m$ and $\mathcal{C}_0^r$ \eqref{eq:C0_det}. In the case $1/2<\rho<1$ that we consider in this theorem, the submanifolds $\mathcal{C}_0^l$ and $\mathcal{C}_0^r$ are normally hyperbolic and therefore persist for positive $\eps$. That is, for sufficiently small $\eps$, there exist submanifolds $\mathcal{C}_\eps^l$ and $\mathcal{C}_\eps^r$ that are invariant under the flow of \eqref{eq:gammas0_slow}/\eqref{eq:gammas0_fast} and $\mathcal{O}(\eps)$ close to $\mathcal{C}_0^l$ and $\mathcal{C}_0^r$.
    \item The critical submanifolds $\mathcal{C}_0^l$ and $\mathcal{C}_0^r$ each have \textbf{stable and unstable manifolds} $\mathcal{W}^{s,u}(\mathcal{C}_0^l)$ and $\mathcal{W}^{s,u}(\mathcal{C}_0^r)$. The dimensions of $\mathcal{W}^{s,u}(\mathcal{C}_0^l)$ and $\mathcal{W}^{s,u}(\mathcal{C}_0^r)$ are determined by the layer problem \eqref{eq:gammas0_layp}. Since every point of $\mathcal{C}_0^l$ is a saddle in the two-dimensional layer problem \eqref{eq:gammas0_layp}, both $\mathcal{W}^{s}(\mathcal{C}_0^l)$ and $\mathcal{W}^{u}(\mathcal{C}_0^l)$ are three-dimensional. The same holds for $\mathcal{W}^{s}(\mathcal{C}_0^r)$ and $\mathcal{W}^{u}(\mathcal{C}_0^r)$. By the normal hyperbolicity of $\mathcal{C}_0^{l,r}$, these stable and unstable manifolds persist for positive and sufficiently small $\eps$ to $\mathcal{W}^{s,u}(\mathcal{C}_\eps^{l,r})$, which are invariant under the flow of \eqref{eq:gammas0_slow}/\eqref{eq:gammas0_fast} and $\mathcal{O}(\eps)$ close to $\mathcal{W}^{s,u}(\mathcal{C}_0^{l,r})$.
    \item System \eqref{eq:gammas0_slow} is \textbf{symmetric under the reflection} $(x,S,U) \to (-x,-S,-U)$, which is a direct consequence of the translational invariance of the PDE system \eqref{eq:model_specific}. System \eqref{eq:gammas0_fast} has the equivalent reflection symmetry $(\xi,S,U) \to (-\xi,-S,-U)$. Both the reduced problem \eqref{eq:gammas0_redpb} and the layer problem \eqref{eq:gammas0_layp} inherit this reflection symmetry. As a result, the reflection map $R: \mathbb{R}^4 \to \mathbb{R}^4$, $R(P,S,T,U) = (P,-S,T,-U)$ is a bijection between $\mathcal{W}^s(\mathcal{C}_\eps^l)$ and $\mathcal{W}^u(\mathcal{C}_\eps^l)$ and between $\mathcal{W}^s(\mathcal{C}_\eps^r)$ and $\mathcal{W}^u(\mathcal{C}_\eps^r)$, both for $\eps=0$ and $0<\eps \ll 1$. Moreover, $\mathcal{C}_0$ \eqref{eq:C0} is invariant under the reflection $R$.
    \item From Proposition \ref{prop:hetconn}, it follows that $\mathcal{W}^s(\mathcal{C}_0^l)$ \textbf{intersects} $\mathcal{W}^u(\mathcal{C}_0^r)$ transversally; this intersection is two-dimensional. By symmetry, the same holds for the intersection of $\mathcal{W}^u(\mathcal{C}_0^l)$ and $\mathcal{W}^s(\mathcal{C}_0^r)$. The transversality of the intersections implies that their perturbed $\eps>0$ counterparts intersect transversally as well, and that intersection therefore has the same dimension. We denote $\mathcal{J}_\text{back} := \mathcal{W}^s(\mathcal{C}_\eps^l) \cap \mathcal{W}^u(\mathcal{C}_\eps^r)$ and $\mathcal{J}_\text{front} := \mathcal{W}^u(\mathcal{C}_\eps^l) \cap \mathcal{W}^s(\mathcal{C}_\eps^r)$. The reflection map $R$ is a bijection between the two-dimensional manifolds $\mathcal{J}_\text{front}$ and $\mathcal{J}_\text{back}$.
    \item Let $\Sigma := \left\{S=0,\,U=0\right\}$ be the hyperplane that is invariant under the reflection map $R$. By uniqueness of solutions, any solution to \eqref{eq:gammas0_slow} with initial conditions in $\Sigma$ has the property that $P$ and $T$ are even functions in $x$ (and, as a direct consequence, $S$ and $U$ are odd functions in $x$). Therefore, any orbit that intersects $\Sigma$ is \textbf{invariant under the reflection map $R$}. 
    \item Our first goal is to show that \textbf{$\mathcal{C}_\eps$ intersects $\Sigma$}. Note that this is immediately seen to be true for $\eps=0$; however, this intersection is \emph{not} transverse since both $\mathcal{C}_0$ and $\Sigma$ are contained in the hyperplane $\left\{U=0\right\}$. Therefore, we cannot use standard transversality arguments to conclude the persistence of the intersection for $\eps>0$.\\
    For sake of clarity and brevity, we write the vector field \eqref{eq:gammas0_fast} as $(\eps\,S,\eps\,F(P,T),U,G(P,T))$ and assume that $\mathcal{C}_0$ can be represented as a graph over the slow variables $(P,S)$ -- in particular, that the condition $G(P,T) = 0$ gives rise to the equality $T = h(P)$. Note that this is true for every separate branch of $\mathcal{C}_0$ \eqref{eq:C0_det}, see also Figure~\ref{fig:CM_cases}. Writing $\mathcal{C}_0 = \left\{\,(P,S,h(P),0)\,\right\}$, standard Fenichel persistence provides the existence of functions $h_1(P,S;\eps)$ and $j_1(P,S;\eps)$ such that $\mathcal{C}_\eps = \left\{(P,S,h(P) + \eps\,h_1(P,S;\eps),\eps\,j_1(P,S;\eps))\right\}$. We now use the fact that $\mathcal{C}_\eps$ is invariant under the flow; therefore, the vectorfield $(\eps\,S,\eps\,F(P,T),U,G(P,T))$ at $\mathcal{C}_\eps$ should be orthogonal to $\left(T \mathcal{C}_\eps\right)^\perp$. Since
    \begin{displaymath}
        \left(T \mathcal{C}_\eps\right)^\perp = \text{span }\left\{\begin{pmatrix}-\eps \frac{\partial j_1}{\partial P}\\-\eps \frac{\partial j_1}{\partial S}\\0\\1\end{pmatrix},\begin{pmatrix}-h'(P)-\eps\frac{\partial h_1}{\partial P}\\-\eps\frac{\partial h_1}{\partial S}\\1\\0\end{pmatrix}\right\},
    \end{displaymath}
    and
    \begin{displaymath}
        \begin{pmatrix}\eps\,S\\\eps\,F(P,T)\\U\\G(P,T)\end{pmatrix}_{(P,S,T,U)\in\mathcal{C}_\eps} = \begin{pmatrix}\eps\,S\\\eps\,F(P,h(P)+\eps h_1(P,S;\eps))\\\eps j_1(P,S;\eps)\\G(P,h(P)+\eps h_1(P,S;\eps))\end{pmatrix},
    \end{displaymath}
    we find in particular that
    \begin{equation}
        j_1(P,S;\eps) = f'(P)S+ \eps S \frac{\partial h_1}{\partial P} - \eps\,\frac{\partial h_1}{\partial S} F(P,h(P)+\eps\,h_1(P,S;\eps)).
    \end{equation}
    From this equation, we see that $j_1(P,-S;\eps) = - j_1(P,S;\eps)$, from which it follows that $j_1(P,0;\eps) = 0$. We conclude that $\mathcal{C}_\eps \cup \left\{S=0\right\} = \left\{(P,0,h(P)+\eps\,h_1(P,0;\eps),\eps\,j_1(P,0;\eps)\right\} \subset \Sigma$.
    \item If a slow-fast periodic $R$-symmetric orbit exists that is close to the singular skeleton, its slow parts lie (exponentially) close to $\mathcal{C}_\eps^l$ and $\mathcal{C}_\eps^r$. Therefore, we take a \textbf{line segment of initial conditions} $A^l \subset \Sigma$, $A_l = \left\{(P_\eps,0,T_\eps,u), u_- < u < u_+\right\}$ with $u_-<0<u_+$ such that $A^l \cap \mathcal{C}_\eps^l =: (P^l_\eps,0,T^l_\eps,0)$ is nonempty and that $(P^l_0,0,T^l_0,0) \in A^l$ for $\eps=0$, where $(P^l_0,0,T^l_0,0)$ is the point where the singular skeleton intersects $\Sigma$.\\
    We study the forward flow of $A^l$ under the full system \eqref{eq:gammas0_fast}. We define the two-dimensional manifold $Z^l$ as the union of all forward orbits with initial conditions in $A_l$. Defining $\gamma(\xi;u)$ as the forward orbit of the initial condition $(P_\eps,0,T_\eps,u) \in A^l$, we can write $Z^l := \bigcup_{u_- < u < u_+} \gamma(\xi;u)$, yielding a natural fibration of $Z^l$. By choosing $|u_\pm|$ sufficiently small, we can ensure that $Z^l$ intersects a neighbourhood of the take-off point $(P_*,S_*,0,0)$.\\
    Due to the normal hyperbolicity of $\mathcal{C}_\eps$, the forward flow of $A^l$ is exponentially contracted towards $\mathcal{W}^u(\mathcal{C}_\eps^l)$. In fact, near $C_\eps^l$, the fast flow is determined by the `left' saddle of the layer problem \eqref{eq:gammas0_layp}. Hence, the fibres of $Z^l$ that remain close to the part of $\mathcal{W}^u(\mathcal{C}_\eps^l)$ may intersect $\mathcal{W}^s(\mathcal{C}_\eps^r)$, i.e. those that `take off' from $\mathcal{C}_\eps^l$ in the direction of positive $T$, are parametrised by $0<u<u_+$; we write $Z^l_+ := \bigcup_{0 < u < u_+} \gamma(\xi;u)$.
    \item Since $Z_l$ intersects a neighbourhood of the take-off point $(P_*,S_*,0,0)$, it follows that $Z^l_+$ remains close to the intersection $\mathcal{W}^u(\mathcal{C}_\eps^l) \cup \mathcal{W}^s(\mathcal{C}_\eps^r)$. In fact, we can use the layer problem \eqref{eq:gammas0_layp} to track the forward evolution of the leading projection of $Z^l_+$ onto $(T,U)$ under the fast flow, \textbf{along the heteroclinic connection}.\\
    We make two observations. First, from the flow of the layer problem, we see that every orbit in $Z^l_+$ transversally intersects the hyperplane $\left\{U=0\right\}$. Furthermore, as $u \downarrow 0$, the orbits $\gamma(\xi;u)$ converge to the heteroclinic orbit, which is forward asymptotic to $\mathcal{C}^r_\eps$ as it is contained in $\mathcal{W}^s(\mathcal{C}_\eps^r)$. Hence, the value $\xi_0$ for which $\gamma(\xi_0;u) \cup \left\{U=0\right\} \neq \emptyset$ can be made arbitrarily large by choosing $u$ arbitrarily small.\\
    The second consequence of $Z^l_+$ being close to $\mathcal{W}^s(\mathcal{C}_\eps^r)$ is that the orbits $\gamma(\xi,u) \in Z^l_+$ follow the slow flow on $\mathcal{C}_\eps^r$ \emph{for some time} after the the jump through the fast field. In other words, there are $\xi_1$ and $\xi_2$ such that $\gamma(\xi,u)$ is exponentially close to $\mathcal{C}_\eps^r$ for $\xi_{\text{td}} < \xi < \xi_{\text{to}}$; moreover, we can make $\xi_{\text{to}}$ arbitrarily large by choosing $u$ arbitrarily small.\\ 
    This allows us to identify two scenarios for orbits contained in $Z^l_+$. First, we choose $u$ sufficiently small, so that $\gamma(\xi;u)$ closely follows the slow flow on $\mathcal{C}^r_\eps$ through the hyperplane $\left\{S=0\right\}$; we define $\xi_1$ to be the $\xi$-value for which $\gamma(\xi_2;u) \cap \left\{S=0\right\} \neq \emptyset$. Then, we can further choose $u$ sufficiently small such that the intersection with the hyperplane $\left\{U=0\right\}$ takes place \emph{after} the intersection with $\left\{S=0\right\}$ -- that is, $\xi_0 > \xi_1$. Second, due to the saddle structure of the fast field near $\mathcal{C}^r_\eps$, we see that the amount of time spent following the slow flow on $\mathcal{C}^r_\eps$, i.e. $\xi_{\text{to}}-\xi_{\text{td}}$ is monotonic in $u$. As a consequence, we can choose $u$ sufficiently large such that the orbit $\xi(\xi;u)$ takes off from (a neighbourhood of) $\mathcal{C}^r_\eps$ \emph{before} the hyperplane $\left\{S=0\right\}$ is intersected -- that is, $\xi_0 < \xi_1$. 
    \item By \textbf{continuity of $\gamma(\xi;u)$ in $u$}, there exists an $u_* \in (0,u_+)$ such that $\xi_0 = \xi_1$, that is, $\gamma(\xi,u_*)$ intersects the reflection hyperplane $\Sigma$ near $\mathcal{C}_\eps^r$. By the reflection symmetry of the full system \eqref{eq:gammas0_fast}, it follows that $\gamma(\xi;u_*)$ is periodic and by construction close to the singular skeleton.
\end{enumerate}
\end{proof}

Note that points 7, 8 and 9 of the proof of Theorem \ref{thm:po_existence} also occur in the proof of \cite[Theorem 3.4]{NDV.2026}.

\subsection{Fast toxicity diffusion}
While our choice to take $0<\eps\ll1$ small was motivated by the parameter choices taken in \cite{GIS_2026}, the case where the toxicity diffusion rate is \emph{large} in comparison to the biomass diffusion rate also falls within the scope of GSPT analysis. Interestingly, this converse asymptotic limit allows for the existence of spike patterns. We sketch the setup of the analysis below, but refrain from going into details, as this is beyond the scope of the current paper.\\

We consider system \eqref{eq:model_specific_dynsys_inP} and assume $\eps \gg 1$. We introduce $0<\delta = 1/\eps \ll 1$, which makes system \eqref{eq:model_specific_dynsys_inP} singularly perturbed. The associated critical manifold $\mathcal{M}_0$ is found in the limit $\delta \downarrow 0$:
\begin{equation}\label{eq:slowtox_C0}
    \hat{\mathcal{M}}_0 := \left\{S = 0,\; \frac{P}{1 - \rho\,\phi(T)}\left[d+s\,\phi(T) - \left(g - \gamma\,\phi(T)\right)\left(1-\frac{P}{1 - \rho\,\phi(T)}\right)\right] = 0\,\right\}.
\end{equation}
We see that $\hat{\mathcal{M}}_0$ can be written as the union of two manifolds, $\hat{\mathcal{M}}_0 = \hat{\mathcal{M}}_1 \cup \hat{\mathcal{M}}_2$,
with
\begin{equation}
 \hat{\mathcal{M}}_1 := \left\{ S = 0,\,P = 0\right\}
\end{equation}
and
\begin{equation}
    \hat{\mathcal{M}}_2 := \left\{S=0,\,d+s\,\phi(T) - \left(g - \gamma\,\phi(T)\right)\left(1-\frac{P}{1 - \rho\,\phi(T)}\right) = 0\right\}.
\end{equation}
On $\hat{\mathcal{M}}_2$, $P$ can be written as a function of $T$:
\begin{equation}\label{eq:slowtox_C0_PinT}
    P(T) = \left(1-\rho\,\phi(T)\right) \left(1-\frac{d + s\, \phi(T)}{g-\gamma\,\phi(T)}\right)= \left\{ \begin{array}{rcl} \left(1-\rho\,T\right) \left(1-\frac{d + s\, T}{g-\gamma\,T}\right) & \text{if} & T<1,\\ \left(1-\rho\right) \left(1-\frac{d + s}{g-\gamma}\right) & \text{if} & T>1.\end{array}\right.
\end{equation}
Depending on the value of $d$, $g$, $s$ and $\gamma$, the branch $\hat{\mathcal{M}}_2$ can cross the trivial branch $\hat{\mathcal{M}}_1$, generically leading to an exchange in stability. When (part of) the $\hat{\mathcal{M}}_2$-branch is located in the biologically feasible $T>0$, $P>0$ quadrant, the associated layer problem will admit a homoclinic solution. This fast homoclinic can be used to construct periodic spike-type patterns, which can be analysed along the lines of \cite{DoelmanVeerman.2015,dRDR.2016}. Previous analysis \cite{DoelmanVeerman.2015} suggests that, for such spike-type patterns to be observable, the parameter $h$ should be asymptotically large in $\delta$, i.e. $h = h_0/\delta$.
The existence of localised pulses follows from \cite[Theorem 2.1]{DoelmanVeerman.2015}; the existence of periodic pulse patterns follows from \cite[Theorem 2.11]{dRDR.2016}. Note that both references also provide the necessary theory to establish the stability of these spike patterns.




\section{Periodic multiscale patterns in the general model 
}\label{sec:patterns_general}

We consider the general model \eqref{eq:model_class}. Using the vector notation $\gamma := (\gamma_1,\gamma_2)$, $\eta := (\eta_1,\eta_2)$ and $\xi(R,T) := \left(\xi_1(R,T),\xi_2(R,T)\right)$, we first write system \eqref{eq:model_class} as
\begin{subequations}\label{eq:model_class_innerproducts}
    \begin{align}
    \partial_t R &= d_{R_1}\Delta R - (d_{R_1}-d_{R_2}) \Delta \xi_2(R,T) + (\hat{R}-R)\left\langle \gamma, \xi(R,T) \right\rangle,\\
    \partial_t T &= d_T \Delta T + \mu \langle \eta,\xi(R,T) \rangle - k\,T,
    \end{align}
\end{subequations}
using the definitions of the nonlinearities \eqref{eq:model_class_nonlinearities}, where $\langle \cdot,\cdot \rangle$ is the usual inner product in $\mathbb{R}^2$. Then, we rescale space $x \to \sqrt{d_{R_1}} x$ and time $t \to t/k$, and subsequently define the parameters
\begin{equation}
    \eps^2 := \frac{d_T}{k d_{R_1}} , \quad \hat{\mu} := \frac{\mu}{k}, \quad \sigma := \frac{d_{R_2}}{d_{R_1}}
\end{equation}
to obtain
\begin{subequations}\label{eq:model_class_innerproducts_rescaled}
    \begin{align}
    \partial_t R &= \Delta R - (1-\sigma) \Delta \xi_2 + (\hat{R}-R)\langle \gamma, \xi(R,T) \rangle,\\
    \partial_t T &= \eps^2 \Delta T + \hat{\mu} \langle \eta,\xi(R,T) \rangle - T.
    \end{align}
\end{subequations}
We introduce
\begin{equation}\label{eq:model_class_def_P}
    P := R - (1-\sigma) \xi_2(R,T) = R - (1-\sigma) R \frac{f(T)}{f(T)+g(T)} = R \frac{\sigma f(T)+ g(T)}{f(T)+g(T)}
\end{equation}
to facilitate the upcoming analysis. Furthermore, we introduce the rescaled interaction functions $\hat{\xi}_i := \xi_i/R$. Note that $\hat{\xi}_1 + \hat{\xi}_2 = 1$, and that $\hat{\xi}_{1,2}$ are functions of $T$ only (cf.~\eqref{eq:model_class_nonlinearities}). Moreover, note that comparing \eqref{eq:model_class_innerproducts_rescaled} with the equivalently rescaled model \eqref{eq:model_specific_rescaled}, we see that $\rho = 1-\sigma$; the assumption $d_{R_1} > d_{R_2}$ is equivalent to $0<\sigma<1$. All other system parameters $\gamma_{1,2}$, $\eta_{1,2}$ and $\hat{\mu}$ are assumed to be nonnegative \cite{Morgan_etal.2025}.\\

We study stationary solutions to \eqref{eq:model_class_innerproducts_rescaled}, and consider one unbounded space dimension, i.e. $x \in \mathbb{R}$. These stationary solutions obey the system of ODEs
\begin{subequations}
    \begin{align}
        P'' &= R\langle \eta, \hat{\xi}(T) \rangle - R(\hat{R}-R) \langle \gamma, \hat{\xi}(T) \rangle,\\
        \eps^2 T'' &= T- \hat{\mu} R \langle \eta,\hat{\xi}(T) \rangle,
    \end{align}
\end{subequations}
cf.~\eqref{eq:model_class_def_P}. Substituting $R = \dfrac{P}{\langle(1,\sigma),\hat{\xi}(T)\rangle}$, we obtain
\begin{subequations}\label{eq:model_class_innerproducts_stationary_P}
    \begin{align}
        P'' &= P\left[\frac{\langle \eta, \hat{\xi}(T) \rangle}{\langle (1,\sigma), \hat{\xi}(T) \rangle} - \frac{\langle \gamma, \hat{\xi}(T) \rangle}{\langle (1,\sigma), \hat{\xi}(T) \rangle}\left(\hat{R} - \frac{P}{\langle(1,\sigma),\hat{\xi}(T)\rangle}\right)\right],\\
        \eps^2 T'' &= T- \hat{\mu} P \frac{\langle \eta, \hat{\xi}(T) \rangle}{\langle (1,\sigma), \hat{\xi}(T) \rangle}.
    \end{align}
\end{subequations}
These equations can be written as a four-component dynamical system, yielding
\begin{subequations}\label{eq:model_class_innerproducts_stationary}
    \begin{align}
        P' &= S,\\
        S' &= P\left[\frac{\langle \eta, \hat{\xi}(T) \rangle}{\langle (1,\sigma), \hat{\xi}(T) \rangle} - \frac{\langle \gamma, \hat{\xi}(T) \rangle}{\langle (1,\sigma), \hat{\xi}(T) \rangle}\left(\hat{R} - \frac{P}{\langle(1,\sigma),\hat{\xi}(T)\rangle}\right)\right],\\
        \eps\,T' &= U,\\
        \eps\,U' &= T- \hat{\mu} P \frac{\langle \eta, \hat{\xi}(T) \rangle}{\langle (1,\sigma), \hat{\xi}(T) \rangle},
    \end{align}
\end{subequations}
cf. \eqref{eq:model_specific_dynsys_inP} / \eqref{eq:gammas0_slow}.\\ 

We consider system \eqref{eq:model_class} in the case of slow toxicity diffusion, i.e. $d_T \ll d_{R_1}$, which means that $0<\eps \ll 1$. Our goal is to determine whether system \eqref{eq:model_class_innerproducts_stationary} admits periodic double front type solutions as those constructed for system \eqref{eq:gammas0_slow} in section \ref{sec:patterns_specific} and described in more detail in Proposition \ref{prop:singsol} and Theorem \ref{thm:po_existence}.\\

To establish the existence of periodic double font type solutions in the general model \eqref{eq:model_class_innerproducts_stationary}, an analysis along the same lines as those presented in Section \ref{sec:patterns_specific} can be carried out. The goal of this section is not to carry out this analysis in full detail; rather, it is to identify \emph{necessary} conditions for periodic double front patterns to occur. In particular, the geometry of the critical manifold associated to system \eqref{eq:model_class_innerproducts_stationary} will play a crucial role.

\subsection{The geometry of the critical manifold}
The critical manifold $\mathcal{C}_0$ that structures the dynamics of system \eqref{eq:model_class_innerproducts_stationary} is defined as
\begin{equation}
    \mathcal{C}_0 := \left\{\,T = \hat{\mu} P \frac{\langle \eta, \hat{\xi}(T) \rangle}{\langle (1,\sigma), \hat{\xi}(T) \rangle} ,\;U = 0\,\right\}.
\end{equation}
For the ecologically relevant case $T>0$, the projection of $\mathcal{C}_0$ onto the $(T,P)$-plane is a graph of the function
\begin{equation}\label{eq:P(T)}
    P(T) = \frac{T}{\hat{\mu}}\frac{\langle (1,\sigma), \hat{\xi}(T) \rangle}{\langle \eta, \hat{\xi}(T) \rangle} = \frac{T}{\hat{\mu}} \frac{\sigma f(T)+g(T)}{\eta_2 f(T)+\eta_1 g(T)}.
\end{equation}
Hence, the shape of the graph of $P(T)$ determines the geometry of $\mathcal{C}_0$. In particular, the monotonicity properties of $P(T)$ determine whether front-type patterns can occur in system \eqref{eq:model_class_innerproducts_stationary_P}. The reason for this is the intimate relation between the geometry of $\mathcal{C}_0$ and the equilibria of the layer problem
\begin{subequations}\label{eq:general_layp}
    \begin{align}
        \dot{T} &= U,\\
        \dot{U} &= T- \hat{\mu} P \frac{\langle \eta, \hat{\xi}(T) \rangle}{\langle (1,\sigma), \hat{\xi}(T) \rangle}.
    \end{align}
\end{subequations}
The number of equilibria of \eqref{eq:general_layp} for fixed $P = P^\ast$ is determined by the number of solutions to the equation $P(T) = P^\ast$ -- that is, the number of times the graph of $P(T)$ \eqref{eq:P(T)} intersects the horizontal line $P = P^\ast$. Moreover, due to the Hamiltonian structure of the layer problem \eqref{eq:general_layp}, the nature of those equilibria is directly determined by the sign of the intersection of the graph of $P(T)$ with the line $P = P^\ast$. That is, given $T^\ast$ is such that $P(T^\ast) = P^\ast$, then the associated equilibrium in \eqref{eq:general_layp} is a saddle if and only if $\frac{\text{d} P}{\text{d}T}(T^\ast) > 0$ and a centre if and only if $\frac{\text{d} P}{\text{d}T}(T^\ast) < 0$. Hence, a heteroclinic orbit in the layer problem can only exist when there are \emph{two} distinct values $T^\ast_\pm$ for which $P(T^\ast_\pm) = 0$ and $\frac{\text{d} P}{\text{d} T}(T^\ast_\pm) > 0$. \\




Since both $f,\, g : [0,\infty) \to [0,\infty)$ and $f$ is monotonically increasing, while $g$ is monotonically decreasing (cf. section \ref{sec:model_class}), we can directly infer that $P(T) > 0$ for $T>0$ and that $P(0) = 0$. Moreover, we see that $\frac{\text{d} P}{\text{d} T}(0) = \frac{1}{\hat{\mu}} \frac{\langle \eta,\hat{\xi}(0)\rangle}{\langle (1,\sigma),\hat{\xi}(0)\rangle} > 0$ and that $\lim_{T \to \infty} \frac{\text{d} P}{\text{d} T} = \frac{\sigma}{\hat{\mu}\eta_2} > 0$ \footnote{Here, we implicitly assume that $g \to 0$ or $f$ is unbounded as $T \to \infty$. If both $\lim_{T \to \infty}g(T) := g_\infty > 0$ and $\lim_{T \to \infty}f(T) := f_\infty < \infty$, then the limit value of $P'$ is $\frac{\sigma f_\infty+g_\infty}{\hat{\mu}(\eta_2 f_\infty + \eta_1 g_\infty}>0$.}. Hence, for any $P^\ast >0$, the graph of $P(T)$ intersects the line $P = P^\ast$ at least once. In addition, by continuity (since both $f$ and $g$ are assumed to be piecewise $C^1$), the graph of $P(T)$ intersects the line $P = P^\ast$ at least twice with positive derivative if and only if there is an open $T$-interval for which $P(T)$ is decreasing, i.e. where $\frac{\text{d}P}{\text{d}T} < 0$. Hence, the existence of a heteroclinic orbit in the layer problem \eqref{eq:general_layp} is only possible when $\frac{\text{d}P}{\text{d}T} < 0$ on some open $T$-interval; see Figure \ref{fig:CM_cases} for an illustration of this observation for the specific model \eqref{eq:model_specific}.\\

To establish the existence of such a $P$-decreasing interval, we assume that $g(T) > 0$, i.e. that $g(T)$ is nontrivial on an open $T$-interval. Indeed, when $g(\hat{T}) = 0$ for some $\hat{T}$, then $g(T) = 0$ for all $T \geq \hat{T}$ by the monotonicity of $g$; this means that $P(T) = T \frac{\sigma}{\hat{\mu} \eta_2}$ for all $T \geq \hat{T}$, and therefore $P'(T)>0$ for all $T \geq \hat{T}$.\\
Next, we introduce $h(T) = \frac{f(T)}{g(T)}$ and rewrite $P(T)$ as $P(T) = \frac{T}{\hat{\mu}} \frac{1+\sigma h(T)}{\eta_1+\eta_2 h(T)}$. We find
\begin{align*}
    \hat{\mu}  \, P'(T) &= \frac{1+\sigma h(T)}{\eta_1+\eta_2 h(T)} + T \left(\frac{\sigma h'(T)}{\eta_1+\eta_2 h(T)} - \eta_2 h' \frac{1+\sigma h(T)}{(\eta_1+\eta_2 h(T))^2}\right)\\
    &= \frac{1}{\eta_1+\eta_2 h(T)} \left[1 + \sigma h(T) +T h'(T)\frac{\eta_1\sigma -\eta_2}{\eta_1+\eta_2 h(T)} \right],
\end{align*}
which is negative if and only if 
\begin{equation}\label{eq:condition_hprime}
    1 + \sigma h+T h'(T)\frac{\eta_1\sigma -\eta_2}{\eta_1+\eta_2 h(T)}<0.
\end{equation}
Since $f'(T) \geq 0$ and $g'(T) \leq 0$, we know that $h'(T) \geq 0$, which means that \eqref{eq:condition_hprime} can hold only if $\eta_1\sigma - \eta_2 < 0$. Assuming therefore that $\eta_1\sigma - \eta_2 < 0$, we can reformulate \eqref{eq:condition_hprime} as a differential inequality for $h(T)$, yielding
\begin{displaymath}
    h'(T) < -\frac{1}{T} \frac{(\eta_1+\eta_2 h(T))(1+\sigma h(T))}{\eta_1\sigma -\eta_2}.
\end{displaymath}
Writing $h(T) = \frac{1}{k(T)} - \frac{\eta_1}{\eta_2}$, this differential inequality can be made linear, yielding
\begin{displaymath}
    k'(T) < -\frac{\eta_2-\eta_1 \sigma}{T}\left(\eta_2 \sigma + (\eta_2-\eta_1 \sigma)k(T)\right),
\end{displaymath}
Gr{\"o}nwall's Lemma now provides the estimate
\begin{displaymath}
    T^{(\eta_2-\eta_1 \sigma)^2}\left(\frac{\eta_2 \sigma}{\eta_2-\eta_1 \sigma} + k(T)\right) < T_-^{(\eta_2-\eta_1 \sigma)^2}\left(\frac{\eta_2 \sigma}{\eta_2-\eta_1 \sigma} + k(T_-)\right),
\end{displaymath}
where $T_-$ is the left limit of the interval where $P(T)$ is decreasing. Reverting to $h(T)$, we obtain
\begin{displaymath}
    h(T) < \frac{(T/T_-)^{(\eta_2-\eta_1 \sigma)^2} \frac{\eta_1 + \eta_2 h(T_-)}{1+\sigma h(T_-)}-\eta_1}{(T/T_-)^{(\eta_2-\eta_1 \sigma)^2} \sigma\frac{\eta_1 + \eta_2 h(T_-)}{1+\sigma h(T_-)}-\eta_2},
\end{displaymath}
which means in terms of the original model functions $f$ and $g$ that
\begin{equation}\label{eq:fg_inequality}
    \frac{f(T)}{g(T)} < \frac{(T/T_-)^{(\eta_2-\eta_1 \sigma)^2} \frac{\eta_1 g(T_-) + \eta_2 f(T_-)}{g(T_-)+\sigma f(T_-)}-\eta_1}{(T/T_-)^{(\eta_2-\eta_1 \sigma)^2} \sigma\frac{\eta_1 g(T_-) + \eta_2 f(T_-)}{g(T_-)+\sigma f(T_-)}-\eta_2}.
\end{equation}
The outcome of the above analysis, and its consequences for the existence of front type patterns in system \eqref{eq:model_class_innerproducts}, is summarised in the following Result:

\begin{res}\label{res:patterns_general_necessaryconditions} Let $f,g : [0,\infty) \to [0,\infty)$ be piecewise $C^1$; assume that $f$ is monotonically increasing and $g$ is monotonically decreasing.
\begin{enumerate}
    \item If either 
    \begin{itemize}
        \item $\eta_1\sigma > \eta_2$, or
        \item $\eta_1\sigma < \eta_2$ but there does not exist an interval $(T_-,T_+) \subset (0,\infty)$ such that for $T \in (T_-,T_+) $, $f$ and $g$ obey the inequality \eqref{eq:fg_inequality},
    \end{itemize}
    then the graph of $P(T)$ \eqref{eq:P(T)} is nondecreasing, hence system \eqref{eq:model_class_innerproducts} does \emph{not} admit stationary front-type patterns.
    \item If $\eta_1\sigma < \eta_2$ and there exists an interval $(T_-,T_+) \subset (0,\infty)$ such that for $T \in (T_-,T_+) $, $f$ and $g$ obey the inequality \eqref{eq:fg_inequality}, then the graph of $P(T)$ is decreasing on $(T_-,T_+)$; hence, system \eqref{eq:model_class_innerproducts} may admit stationary front-type patterns.
\end{enumerate}
    
\end{res}

 \subsection{Stationary front-type patterns for canonical switching rates}
 We investigate whether stationary front-type patterns can exist for the canonical choices for the switching rate functions $f$ and $g$ as listed in section \ref{sec:model_class}, namely power functions \eqref{eq:fg_choices_powerfct}, Holling-type II functions \eqref{eq:fg_choices_HollingII} and saturation functions \eqref{eq:fg_choices_sattrans}.

 \subsubsection{Saturation functions} For $f$ and $g$ as in \eqref{eq:fg_choices_sattrans}, we find that 
 \begin{equation}
  P_\text{sat}(T) = \left\{\begin{array}{rcl} \dfrac{T}{\hat{\mu}}\dfrac{1-\frac{T}{\hat{T}}(1-\sigma)}{\eta_1 + \frac{T}{\hat{T}}(\eta_2-\eta_1)} & \text{if} & T < \hat{T},\\ \dfrac{ \sigma T}{\hat{\mu}\eta_2} & \text{if} & T > \hat{T} .   
  \end{array}\right.    \,
 \end{equation}
 The derivative of $P_\text{sat}$ can be calculated directly, from which it follows that there is an isolated, positive $T$-interval where $P_\text{sat}'(T) <0$ when $\eta_{1,2}>0$ and $0 \leq \sigma \eta_1 < (1-\sigma)\eta_2$; this interval is given by
\begin{displaymath}
    T/\hat{T} \in \left(\frac{1}{\frac{\eta_2}{\eta_1}-1} \left[-1+\sqrt{1+\frac{\frac{\eta_2}{\eta_1}-1}{1-\sigma}}\right],1\right).
\end{displaymath}
Note that the saturation function choice of $f$ and $g$ as in \eqref{eq:fg_choices_sattrans} produces --barring some simplifying assumptions-- the specific model \eqref{eq:model_specific} that was studied in section \ref{sec:patterns_specific}. In this context, the condition $0 \leq \sigma \eta_1 < (1-\sigma)\eta_2$ is equivalent to $1/2 < \rho < 1$ (see Proposition \ref{prop:hetconn}), where the equivalent $P$-decreasing interval is given by $\left(1/(2\rho),1\right)$; see also Figure \ref{fig:CM_cases}, recalling the definition of $\mathcal{C}_0^m$ \eqref{eq:C0_det}. The detailed analysis carried out in Section \ref{sec:patterns_specific} and culminating in Theorem \ref{thm:po_existence} provides conditions for which periodic front-type patterns are rigorously shown to exist, strengthening --and going well beyond-- statement 2 of Result \ref{res:patterns_general_necessaryconditions}.

\subsubsection{Power functions} For $f$ and $g$ as in \eqref{eq:fg_choices_powerfct}, we find that 
\begin{equation}\label{eq:PpowerT}
    P_\text{power}(T) = \frac{T}{\hat{\mu}} \frac{1+\sigma T^p(1+T)^q}{\eta_1 + \eta_2 T^p (1+T)^q},
\end{equation}
 with $p,q > 0$.
 The derivative of $P_\text{power}$ can be calculated directly, from which it follows that there is an isolated, positive $T$-interval where $P_\text{sat}'(T) <0$ if and only if there exist (at least) two positive $T$-values for which
\begin{align*}
 (1+T) \eta_1 + F(T)\left[(1-p)\eta_2 + (1+p)\eta_1 \sigma + \left((1+p+q)\eta_1 \sigma +(1-p-q)\eta_2\right)T\right]\\ + \eta_2 \sigma (1+T) F(T)^2 = 0,
\end{align*}
i.e.
\begin{equation}\label{model_class_powerfct_condition_eq}
 F(T) = T^p (1+T)^q  = \frac{a + b T}{2 \eta_2 \sigma (1+T)}\left[-1 \pm \sqrt{1-\frac{4\eta_1 \eta_2 \sigma(1+T)^2}{(a+b T)^2} }\right],
\end{equation}
with $a = (1-p)\eta_2 + (1+p)\eta_1 \sigma$ and $b=(1+p+q)\eta_1 \sigma +(1-p-q)\eta_2$. If $a+b T > 0$, then $-1 \pm \sqrt{1-\frac{4\eta_1 \eta_2 \sigma(1+T)^2}{(a+b T)^2} }<0$ so \eqref{model_class_powerfct_condition_eq} does not have solutions for positive $T$. Hence, at least one of $a$, $b$ must be negative.\\
Writing $y(T) = \frac{a + b T}{2 \eta_1 \sigma(1+T)}$, we see that $y(T) \to a/(2 \eta_1 \sigma)$ as $T \downarrow 0$ and $y(T) \to b/(2\eta_1 \sigma)$ as $T \to \infty$. In addition, the right hand side of \eqref{model_class_powerfct_condition_eq}, which is expressed in terms of $y(T)$ as $y(T) \left[-1 \pm \sqrt{1-\frac{\eta_2}{\eta_1 \sigma} y(T)^{-2}}\right]$, is positive if and only if $y(T)<0$ and the $\pm$ sign is chosen to be negative.\\
Since $a<0$ or $b<0$, there are three possible combinations to explore. We first consider the case $a<0$ and $b>0$. Since $b = a + q(\eta_1 \sigma - \eta_2)$, this means that $\eta_1 \sigma - \eta_2 > 0$. However, $a = (\eta_1 \sigma -\eta_2 )p + \eta_1 \sigma + \eta_2 > 0$, which leads to a contradiction. Therefore, we find that necessarily $b<0$. Since $b = \eta_1 \sigma + \eta_2 + (p+q)(\eta_1 \sigma - \eta_2)$, we see that $b$ is negative only if 
\begin{displaymath}
\sigma\eta_1 < \eta_2,    
\end{displaymath}
a necessary condition for the solvability of \eqref{model_class_powerfct_condition_eq}; see also Result \ref{res:patterns_general_necessaryconditions}. The existence of a solution pair for \eqref{model_class_powerfct_condition_eq}, which yields an open positive $S$-interval for which $P_\text{power}'(T)<0$, is now a matter of choosing $\eta_{1,2}$, $\sigma$ and $p$, $q$ appropriately. We refrain from going into a detailed analysis here, but observe that the choice $p=1$, $q=2$, $\eta_1=1$, $\eta_2 = 2$ and $\sigma = \frac{1}{20}$ yields a nontrivial positive solution pair to \eqref{model_class_powerfct_condition_eq}; see also Figure \ref{fig:front_powerfct_pars}.\\

\begin{figure}
    \begin{minipage}{.4\textwidth}
    \centering
    \begin{overpic}[width=1.2\textwidth]{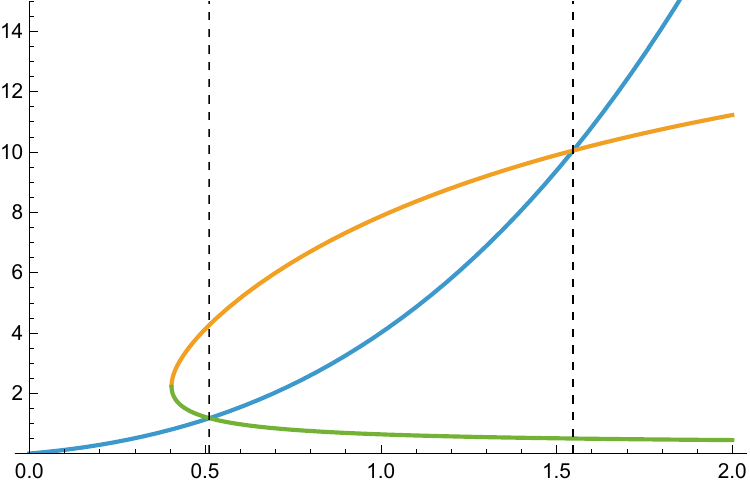}
    \put(102,3){$T$}
    \end{overpic}\\
    (a)
    \end{minipage}
    \hspace{1.5cm}
    \begin{minipage}{.4\textwidth}
    \centering
    \begin{overpic}[width=1.2\textwidth]{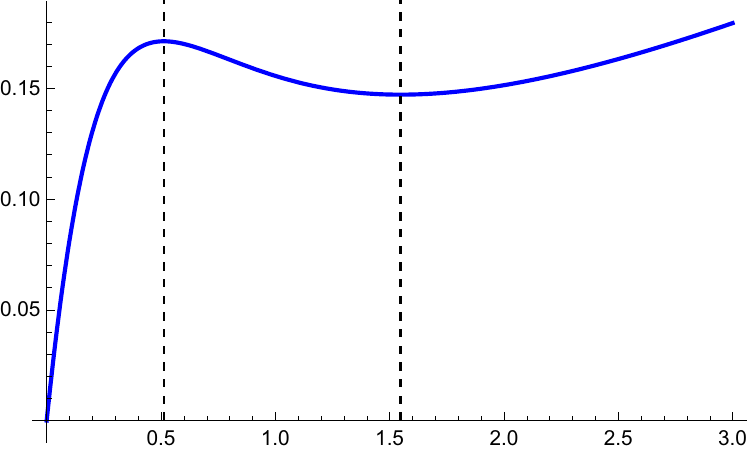}
    \put(103,3){$T$}
    \put(4,62){$P$}
    \put(70,40){\textcolor{blue}{$P_\text{power}(T)$}}
    \end{overpic}\\
    (b)
    \end{minipage}
    \caption{(a) The left- and right hand side of \eqref{model_class_powerfct_condition_eq} are plotted in blue resp. orange (+) and green (-) for $p=1$, $q=2$, $\eta_1 = 1$, $\eta_2 = 2$ and $\sigma = 1/10$. (b) The profile of the critical manifold determined $P_\text{power}(T)$ \eqref{eq:PpowerT} plotted for the same parameter values. The interval where $P_\text{power}'(T)<0$, numerically approximated as $0.5125 < T < 1.5476$, can clearly be observed.}
    \label{fig:front_powerfct_pars}
\end{figure}

\subsubsection{Holling-type II functions}
We take the general form of the Holling-type II functions \eqref{eq:fg_choices_HollingII_rescaled}. Since $c_g>0$, we can eliminate one parameter by rescaling $c_g T =: \hat{T}$, yielding
    \begin{subequations}
        \begin{align}
         \hat{f}(\hat{T}) &= f_0 + \frac{\hat{d}_f \hat{T}}{\hat{c}_f \hat{T}+1},\\
         \hat{g}(\hat{T}) &= g_0 - \frac{\hat{d}_g \hat{T}}{\hat{T}+1},
        \end{align}
    \end{subequations}
with $\hat{c}_f := c_f/c_g$, $\hat{d}_f := d_f/c_g$ and $\hat{d}_g := d_g/c_g$. Note that the codomain $[0,\infty)$ of $\hat{g}$ implies that $\hat{d}_g \leq g_0$.\\
The condition
\begin{displaymath}
    \frac{\text{d}}{\text{d} T} P_\text{Holling}(T) = \frac{\text{d}}{\text{d} \hat{T}} \hat{T} \frac{\sigma \hat{f}(\hat{T}) +\hat{g}(\hat{T})}{\eta_2 \hat{f}(\hat{T}) + \eta_1\hat{g}(\hat{T})} < 0
\end{displaymath}
yields a condition on a fourth order polynomial of the form
\begin{equation}\label{eq:HollingII_quartic_cond}
    \hat{a}_4 \hat{T}^4 + \hat{a}_3 \hat{T}^3 + \hat{a}_2 \hat{T}^2 + \hat{a}_1 \hat{T} + \hat{a}_0 < 0,
\end{equation}
with coefficients
\begin{subequations}
    \begin{align}
    \hat{a}_4 &= \left(\eta_1 \hat{c}_f(g_0 - \hat{d}_g) + \eta_2 (\hat{c}_f f_0 + \hat{d}_f)\right) \left(\sigma(\hat{c}_f f_0 + \hat{d}_f) + \hat{c}_f(g_0 - \hat{d}_g)\right),\\
    \hat{a}_3 &= 2 \left(\eta_1(g_0 - \hat{d}_g + \hat{c}_f g_0) +\eta_2(f_0 + \hat{d}_f + \hat{c}_f f_0)\right)\left(\sigma (\hat{c}_f f_0 + \hat{d}_f) +\hat{c}_f(g_0- \hat{d}_g)\right),\\
    \hat{a}_1 &= 2 \left(\eta_1 g_0 + \eta_2 f_0\right)\left(\sigma(f_0 + \hat{d}_f) + g_0 - \hat{d}_g + \hat{c}_f(\sigma f_0 + g_0)\right),\\
    \hat{a}_0 &= \left(\eta_1 g_0 + \eta_2 f_0\right)\left(\sigma f_0 + g_0\right).
    \end{align}
\end{subequations}
The coefficients $\hat{a}_4, \hat{a}_3$, $\hat{a}_1$ and $\hat{a}_0$ are nonnegative, including the `classic' Holling-type II case for which $f_0 = 0$ and $\hat{d}_g = g_0$. As a consequence, condition \eqref{eq:HollingII_quartic_cond} can only hold for a positive $\hat{T}$-interval if the fourth order polynomial $\hat{a}_4\hat{T}^4 + \hat{a}_3 \hat{T}^3 + \hat{a}_2 \hat{T}^2 + \hat{a}_1 \hat{T} + \hat{a}_0$ has two extrema for positive $\hat{T}$, i.e. two positive $\hat{T}$-values for which
\begin{displaymath}
    12 \hat{a}_4\hat{T}^2 + 6 \hat{a}_3 \hat{T} + 12 \hat{a}_2 = 0.
\end{displaymath}
Since in particular $\hat{a}_3>0$, we see that this quadratic equation cannot have two positive solutions. Hence, there is no positive $T$-interval for which $P_\text{Holling}'(T) < 0$, from which it follows that front type patterns cannot exist for the Holling-type II choice of interaction functions $f$ and $g$ \eqref{eq:fg_choices_HollingII_rescaled}.\\

The analysis of the previous subsections is summarised in the following Result:
\begin{res}\label{res:patterns_generalmodel}
    Let $f,g: [0,\infty) \to [0,\infty)$ be chosen as power functions \eqref{eq:fg_choices_powerfct}, Holling-type II functions \eqref{eq:fg_choices_HollingII_rescaled} or saturation functions \eqref{eq:fg_choices_sattrans}.
    \begin{itemize}
     \item For power functions, the geometry of the critical manifold may be such that front-type patterns can exist in model \eqref{eq:model_class}.
     \item For Holling-type II functions, model \eqref{eq:model_class} does not support front-type patterns.
     \item Front-type patterns may exist in model \eqref{eq:model_class} for saturation transition functions, the specific system of section \ref{sec:model_specific} being an example.
    \end{itemize}
\end{res}

\section{Conclusion}


In this work, we have investigated the existence of far-from-equilibrium stationary patterns in a class of vegetation--autotoxicity models with cross-diffusion using geometric singular perturbation theory. For the specific model \eqref{eq:model_specific}, we have shown the existence of stationary, periodic, front-type far-from-equilibrium patterns, whose singular structure is described in Proposition \ref{prop:singsol} and whose persistence for $0<\varepsilon \ll 1$ is stated in Theorem \ref{thm:po_existence}. For the more general model class \eqref{eq:model_class}, we have identified the properties of the transition rate functions $f(T)$ and $g(T)$ that allow the same geometric construction to be carried out, which are summarised in Result \ref{res:patterns_general_necessaryconditions}. For three biologically relevant choices for the transition rate functions (saturation type, power functions, and Holling-type II), we have investigated whether front-type patterns can exist: Result \ref{res:patterns_generalmodel} summarises the outcome of this investigation.

Our existence analysis highlights the central role of the geometry of the critical manifold in determining whether front-type patterns can exist -- in particular, whether the fast subsystem admits the heteroclinic connection required for the construction of front-type patterns. We note that, rather than being a feature of the specific cross-diffusion model presented in~\cite{GIS_2026}, these patterns originate from specific functional properties of the nonlinearities. In particular, the relation between the cross-diffusion parameter $\sigma$ and the nonlinearity parameters $\eta_{1,2}$ as presented in Result \ref{res:patterns_general_necessaryconditions} is interesting for applications, since it is a clear and easy condition to check.

From a methodological perspective, this work represents --to our knowledge-- the first application of geometric singular perturbation theory to reaction--cross-diffusion systems. While GSPT has become a standard tool for the constructive analysis of far-from-equilibrium patterns in classical reaction--diffusion models, extending these techniques to systems with cross-diffusion presents additional challenges. In particular, the algebraic constraint introduced by the cross-diffusion term requires more careful analysis. We expect that the approach developed here can serve as a starting point for the analysis of patterns in other cross-diffusion models arising in ecology and other application areas.

Our results suggest several directions for future research. It would be interesting to investigate the stability of the resulting periodic patterns and to use numerical continuation to investigate the shape and existence of patterns beyond the realm of (asymptotically) small $\varepsilon$. In addition, the parameter ranges predicted by our analysis can serve as input for more extensive numerical simulations of the PDE system. In terms of biological relevance, the functional properties of the transition rates $f$ and $g$ that either allow for or prevent front-type patterns to occur can be linked to the underlying biological processes. Furthermore, our work is a first step towards understanding structural properties of more general classes of systems containing cross-diffusion operators, which can be understood in more detail through the structure of the associated critical manifold(s). In addition, it would be interesting to investigate to what extent our analytical predictions hold for front-type patterns in two spatial dimensions. 

%

\section*{Acknowledgements}
AI and CS are members of the INdAM-GNFM Group. AI was partially supported by PRIN 2022 PNRR P2022WC2ZZ ``A multidisciplinary approach to evaluate ecosystems resilience under climate change''. CS has received funding from the Programma Giovani Ricercatori ``Rita Levi Montalcini'' 2021.


\bibliography{Crossdiff_references}
\bibliographystyle{test4}

\end{document}